\documentclass[oneside,leqno]{amsart}

\usepackage[T1]{fontenc}
\usepackage{lmodern,times}
\usepackage{microtype} 
\usepackage[usenames,dvipsnames]{xcolor}
 \usepackage{hyperref}
 \hypersetup{colorlinks, citecolor=Violet,
 linkcolor=NavyBlue, urlcolor=Blue}

\usepackage[left=1.5in,right=1.5in, top=1.3in]{geometry}

\usepackage{amsmath,amsfonts,amssymb,enumitem,mathtools}

\numberwithin{equation}{section}
\theoremstyle{plain}                
\newtheorem{theorem}{Theorem}[section]
\newtheorem{lemma}[theorem]{Lemma}
\newtheorem{proposition}[theorem]{Proposition}
\newtheorem{corollary}[theorem]{Corollary}

\theoremstyle{definition}           
\newtheorem{definition}[theorem]{Definition}
\newtheorem{example}[theorem]{Example}

\theoremstyle{remark}
\newtheorem{remark}[theorem]{Remark}

\DeclareMathOperator\Var{Var}

\newcommand{\tot}{\tfrac{1}{2}} 
\newcommand{\oo}[1]{\tfrac{1}{#1}}
\newcommand{\scl}[2]{\langle #1,#2 \rangle} 

\newcommand{\sabs}[1]{| #1 |} 
\newcommand{\abs}[1]{\left| #1 \right|} 
\newcommand{\babs}[1]{\big| #1 \big|} 
\newcommand{\Babs}[1]{\Big| #1 \Big|} 

\newcommand{\set}[1]{\{#1\}} 
\newcommand{\sets}[2]{\set{#1\,:\,#2}} 
\newcommand{\Bset}[1]{\Big\{#1\Big\}} 
\newcommand{\Bsets}[2]{\Bset{#1\,:\,#2}} 

\newcommand{\ind}[1]{ {\mathbf 1}_{{#1}}} 
\newcommand{\inds}[1]{ {\mathbf 1}_{\set{#1}}} 
\newcommand{\seq}[1]{\set{#1_n}_{n\in\N}} 
\newcommand{\seqk}[1]{\set{#1_k}_{k\in\N}} 

\newcommand{\sq}[2][n]{\{ #2 \}_{#1\in\N}}  
\newcommand{\sqk}[1]{\{ #1 \}_{k\in\N}}  
\newcommand{\norm}[1]{{||#1||}} 
\newcommand{\bnorm}[1]{{\big|\big|#1\big|\big|}} 
\newcommand{\lnorm}[1]{{\left\|#1\right\|}} 

\newcommand{\ft}[2]{#1\dots#2}
\renewcommand{\ft}[2]{#1,\dots,#2}

\newcommand{\prf}[1]{ \{ #1 \}_{t\in [0,T]}}

\newcommand{\downto}{\searrow}

\providecommand{\R}{} \renewcommand{\R}{{\mathbb R}}

\newcommand{\N}{{\mathbb N}}
\newcommand{\PP}{{\mathbb P}}

\newcommand{\EE}{{\mathbb E}}
\newcommand{\FF}{{\mathcal F}}

\newcommand{\ee}[1]{ \bE \left[ #1 \right] }

\newcommand{\Bee}[1]{ \bE \Big[ #1 \Big] }

\newcommand{\Beec}[2]{ \Bee{#1 \Big| #2} }

\newcommand{\eeq}[2]{ \bE^{#1}\left[ #2 \right] }

\newcommand{\Beeq}[2]{ \bE^{#1}\Big[ #2 \Big] }

\newcommand{\eecq}[3]{ \bE^{#1}\left[ #2 | #3 \right] }

\newcommand{\izt}{\int_0^t}
\newcommand{\izT}{\int_0^T}

\newcommand{\FFF}{{\mathbb F}}

\newcommand{\eps}{\varepsilon}
\newcommand{\ld}{\lambda}

\newcommand{\vp}{\varphi}

\newcommand{\el}{{\mathbb L}} 

\newcommand{\lone}{\el^1}
\newcommand{\ltwo}{\el^2}
\newcommand{\lpee}{\el^p}

\newcommand{\tolone}{\stackrel{\lone}{\rightarrow}}

\newcommand{\define}[1]{{\textbf{#1}}}

\newcommand{\suppar}[2]{#1^{(#2)}}

\newcommand{\Yn}{\suppar{Y}{n}}
\newcommand{\Ywn}{Y^{[n]}}

\newcounter{notecounter}

\newcommand{\efor}{\text{ for }}
\newcommand{\eforall}{\text{ for all }}
\newcommand{\eand}{\text{ and }}
\newcommand{\ewhere}{\text{ where }}

\newcommand{\cd}{c\` adl\` ag} 

\newcommand\sA{{\mathcal A}}

\newcommand\sB{{\mathcal B}}

\newcommand\sD{{\mathcal D}}

\newcommand\bE{{\mathbb E}}

\newcommand\sF{{\mathcal F}}

\newcommand\sM{{\mathcal M}}

\newcommand\sN{{\mathcal N}}

\newcommand\sT{{\mathcal T}}

\newcommand\sU{{\mathcal U}}

\newcommand\sV{{\mathcal V}}

\newcommand{\sDm}{\sD_{\uparrow}}
\newcommand{\can}{\sD^{d+k}_{\cdot, \uparrow}}
\newcommand{\cano}{\sD^{d+1}_{\cdot, \uparrow}}

\newcommand{\sDf}{\sD_{fv}}
\newcommand{\An}{A^{(n)}}
\newcommand{\Awn}{A^{[n]}}
\newcommand{\Bn}{B^{(n)}}

\newcommand{\bFl}{\bar{\sF}^{L}}

\newcommand{\Fn}{F^{(n)}}
\newcommand{\Fwn}{F^{[n]}}

\newcommand{\Ln}{L^{(n)}}
\newcommand{\Mn}{M^{(n)}}
\newcommand{\Mwn}{M^{[n]}}

\newcommand{\Twn}{T^{[n]}}
\newcommand{\Vn}{\tilde{V}^{[n]}}

\newcommand{\Nn}{N^{(n)}}
\newcommand{\Nwn}{N^{[n]}}

\newcommand{\Pn}{\PP^{(n)}}
\newcommand{\tPn}{\tPP^{(n)}}

\newcommand{\hPP}{\hat{\PP}}
\newcommand{\hPPn}{\hPP^{(n)}}

\newcommand{\ltp}{\ltwo([0,T]\times \Omega, \prog)}
\newcommand{\prob}{{\mathfrak P}}
\newcommand{\prog}{\mathrm{Prog}}
\newcommand{\sAn}{\mathcal A^{[n]}}

\newcommand{\sun}{\sU^{[n]}}

\newcommand{\tPP}{\tilde{\PP}}

\newcommand{\tomz}{\stackrel{\mathrm{MZ}}{\rightarrow}}
\newcommand{\topp}{\stackrel{\mathrm{pp}}{\rightarrow}}

\newcommand{\un}{u^{[n]}}

\newcommand{\ofl}{\otimes_{L}}
\newcommand{\da}{\partial}

\newcommand{\op}[1]{ {}^{\circ}\! #1}
\newcommand{\opp}[1]{ {}^{\circ}\!\! #1}

\newcommand{\canp}{\prob_{\cdot, \uparrow}^{d+k}}
\newcommand{\sAwn}{\sA^{[n]}}
\newcommand{\Vwn}{V^{[n]}}

\newcommand{\as}{\text{a.s.}}
\newcommand{\Leb}{\text{Leb}}
\newcommand{\oPP}{\op{\,\PP}}
\newcommand{\oA}{\opp{A}}

\newcommand{\Ae}{A^{\eps}}

\title%
[The Monotone-Follower Problem]%
{Existence, Characterization and Approximation in the Generalized Monotone-Follower Problem}

\author{Jiexian Li}
\address{Jiexian Li, Department of Mathematics\\
The University of Texas at Austin}
\email{jxli@math.utexas.edu}
\thanks{
  \emph{Acknowledgements:}
  Both authors would like to thank Ioannis Karatzas and Mihai S\^\i rbu for
  valuable conversations, and the anonymous referee for a
  simplification 
  of the proof of Theorem 2.7 and other insightful comments.
  The second author, furthermore, acknowledges the support
  by the National Science Foundation under Grants No. DMS-0706947 (2010 - 2015),
  No.~DMS-1107465 (2012 - 2017) and No.~DMS-1516165 (2015-2018). Any
  opinions, findings and
  conclusions or recommendations expressed in this material are those of the
  author(s) and do not necessarily reflect the views of the National Science
  Foundation (NSF)}
\author{Gordan \v Zitkovi\' c}
\address{Gordan \v Zitkovi\' c, Department of Mathematics\\
The University of Texas at Austin}
\email{gordanz@math.utexas.edu}

\begin{document}



\date{\today}

\begin{abstract}
  We revisit the classical monotone-follower problem and consider it in a
  generalized formulation. Our approach is based on a compactness
  substitute for nondecreasing processes, the Meyer-Zheng weak convergence,
  and the maximum principle of Pontryagin. It establishes existence under
  weak conditions, produces general approximation results and further
  elucidates the celebrated connection between 
  singular stochastic control
  and stopping.
  \end {abstract}

\maketitle


\section{Introduction}
\label{sec:intro}

A direct precursor to the monotone-follower problem dates back to the 1970's;
the basic model originated from engineering and first appeared
in the
work of Bather and Chernoff \cite{BatChe67}. There, it was posed in a
model  of a spaceship being steered towards a target with both precision and
fuel consumption appearing in the performance criterion. The authors observed an
unexpected connection between the control problem they studied and a Brownian
optimal stopping problem based on the same ingredients; arguing quite
incisively, but mostly on heuristic grounds, they demonstrated that the
value function
of the latter is the gradient of the value function of the former.

In 1984, Karatzas and Shreve \cite{KarShr84} considered a generalized version of
the Bather-Chernoff problem dubbing it the ``monotone follower problem''. In
the same paper, using purely probabilistic tools, they established
rigorously the equivalence of the control and stopping problems under
appropriate continuity and growth conditions.  Some time later, Haussmann and
Suo \cite{HauSuo95a} applied relaxation and compactification methods, used the
Meyer-Zheng convergence, and showed existence of the optimal control under a
different set of conditions.  In 2005, Bank \cite{Ban05} constructed a fairly
explicit control policy under stochastic dynamic fuel constraint in one dimension.
Subsequently, Budhiraja and
Ross \cite{BudRos06} applied the Meyer-Zheng convergence to prove a general
existence theorem, also under a fuel constraint.  Guo and Tomecek \cite
{GuoTom08}
 generalized some
results of \cite{KarShr84} in a different direction: they established a connection
between singular control of finite variation and optimal switching.

\medskip

\paragraph{\bf Problem formulation.}
The essence of the   monotone follower problem is
tracking, as closely
as possible, a given random process $L$ (the \emph{target}) by a suitably
constrained control process $A$ (the \emph{follower}). In the original setting
of \cite{KarShr84}, the target is a Brownian motion, the follower is required to
be adapted and non-decreasing, and the ``closeness'' is measured by applying an
appropriate functional to the state variable defined as the difference between
the
position of the target and the position of the follower. Our  version of this
problem is generalized in two directions:

\qquad (a)\ We allow the dynamics of both the target and the follower to be
multidimensional and impose weak assumptions on the distribution of
dynamics the target $L$. For our existence and characterization results
(Theorems \ref{thm:exist} and \ref{thm:max-equiv} below), we only require
that $L$ has \cd {} paths.
For the approximation (Theorem \ref{thm:MZ} below),
we need $L$ to be a Feller process (still allowing, in
particular, inhomogeneities in the cost structure). Also, we consider
functionals which are functions of the target and the follower,
convex in the position of the follower, and not only functions of their relative
positions. Finally, we relax some of the growth assumptions; in particular, we
do not require superlinear growth
  of the cost function to
obtain existence of an optimal control (as in, e.g.,  \cite{Ban05},
where it serves as a sufficient
  condition for the existence of a solution to a stochastic representation
  problem which, in turn, characterizes the optimizer.)

\qquad (b) \ Our formulation is
weak (distributional), in the sense that we are only interested in the joint
distribution of the follower and the target, without fixing the underlying
filtered probability space and making it a part of the problem. This
enables us to prove an approximation result (Theorem \ref{thm:MZ} below)
in great generality. On the other hand,
 as we will see below in
Proposition
\ref{pro:optional}, every weak (distributional) solution can be turned into a
strong one under usually met conditions by a simple projection operation.
Moreover, as far as generality is concerned, any setup where the filtration
is generated by a finite number of \cd{} processes can be easily lifted to our
canonical framework, allowing us to work with on a canonical (Skorokhod) space
right from the start. It is worth noting that (even though we do not provide details
for such an approach here) even greater generality can be achieved by
considering Polish-space-valued \cd{} processes and their natural filtrations.

\paragraph{\bf Our results.} We treat questions of
existence,
approximability and characterization (via Pontryagin's maximum principle), as
well as
connections with optimal stopping. These are tackled using
a variety of methods, including a compactness substitute for monotone
processes and the Meyer-Zheng convergence. Moreover, we posit
the idea that
\begin{quote}\em
the connection between control and stopping can be
understood as the connection between the monotone-follower problem and its
Pontryagin maximum principle.
\end{quote}

The original impetus for our research was twofold:

\qquad (a)\  On the one hand, we
wanted to understand the role played by different regularity and
growth conditions
imposed in the existing literature in order to establish existence of optimal
controls.
This lead to an  existence proof (Theorem \ref{thm:exist} below)  under
less restrictive conditions on most ingredients. The proof
is based on a convenient substitute for compactness under convexity, and not
on the
Meyer-Zheng topology as in some of the papers mentioned above.
The beginnings of such an approach can be traced back to the fundamental result
of Koml\'os
\cite{Sch86}, while the version used in the present paper is due to
Kabanov \cite{Kab99}.

\qquad (b) On other hand - perhaps more importantly -  we tried to grasp a
more practical issue better, namely, the approximation of the
archetypically singular monotone-follower problem by a sequence of regular,
absolutely continuous (even Lipschitz) control problems.  To accomplish
this task, the following conceptual framework was devised. First, a
sequence of so-called ``capped'' problems where the exerted controls are
constrained to be Lipschitz is posed. These \emph{regular} problems come
with increasing upper bounds on the Lipschitz constant and are expected to
approach the monotone control problem both in value and in optimal
controls.  Being regular and well-behaved, each capped problem is expected
to be solvable by the well-known classical methods;  the resulting solution
sequence is, then, expected to converge (in the appropriate sense) towards
the solution to the original problem.

The second, larger, part of the paper can be seen as the implementation
of the above steps. The major difficulty we encountered was the lack of good
equicontinuity estimates on the solutions to the capped problems. To overcome it
we replaced the usual weak convergence under the Skorokhod topology with the
versatile Meyer-Zheng convergence. Even so, we still needed to close the gap
between the limit of the values of the capped problems and the value of the
original problem. For that, we characterized the optimizers (both in the
capped and the original problems) via the maximum
principle of Pontryagin (i.e., the ``first-order'' condition) and passed to a Meyer-Zheng limit there.

While ideas described in the previous paragraphs seem to be new,
the research relating Pontryagin's maximum principle to  singular
control problems is certainly not.  Indeed, the Pontryagin's maximum principle
for singular control problems was first discussed by Cadenillas and Haussmann
\cite{CadHau94} already in 1994. With Brownian dynamics, convex cost, and state
constraints assumed, these authors formulated the stochastic maximum principle
in an integral form and gave necessary and sufficient conditions for optimality.
In order to solve the approximation problem via maximum principle, however, one
must
go beyond their work.
Even though the last
20 years have seen an explosion in activity in the general theory of BSDE and
FBSDE (see e.g., \cite{MaProYon94}, \cite{CviMa96}, \cite{MaJon99},
\cite{MaCvi01}, \cite{AntMa03}, \cite{MaZha11}), to the best of our knowledge
none of the existing work seems to be able to deal directly with the singular
FBSDE that the maximum principle for the monotone-follower problems yields, even
in the Brownian case.  Our route, via approximation and simultaneous
consideration of the related (capped) control problems, can be interpreted as a
variational approach to a class of singular FBSDE and may, perhaps, be of use in
other situations, as well. For example, a combination (see Corollary \ref
{cor:FBSDE} below) of
our existence and
characterization results, i.e., Theorems \ref{thm:exist} and \ref{thm:max-equiv},
guarantees existence of solutions of such FBSDE under weak, monotonicity- and
exponential-growth-type
assumptions on the nonlinearities.

The approximation result (Theorem \ref{thm:approx} below) serves as a
pleasant justification of singular controls as a conceptual limit of
absolutely continuous controls.  Moreover, together with the related
maximum-principle characterization of the optimal controls in the original
problem, it leads us to view the celebrated connection between stopping and
control in a new light.  Indeed, once such a characterization is
formulated, it is a simple observation that it can be re-interpreted as an
optimal stopping problem, which turns out to be precisely the
optimal
stopping problem identified by Bather and Chernoff and rigorously studied
by Karatzas and Shreve.

\medskip

\paragraph{\bf Organization of the paper.}  After this Introduction, Section
2.~contains the formulation of the problem, a description of the probabilistic
setup it is defined on, and main results. Section 3.~is devoted to proofs. At
the end,  a short compendium of the most important well-known results -
including the tightness criteria - on the Meyer-Zheng topology is given in
Appendix A.

\section{The Problem and the Main Results} \label{sec:ingr}
\subsection{Notational conventions and the canonical setup}
For $N\in\N$, let $\sD^N$ denote the Skorokhod space, i.e., the  measurable
space of all $\R^N$-valued \cd{} functions on $[0,T]$, equipped with the
$\sigma$-algebra generated by the coordinate maps. Since the same
$\sigma$-algebra appears as the Borel $\sigma$-algebra generated by the
Skorokhod topology, as well as by most of the other popular topologies on
$\sD^N$, we call it simply the Borel $\sigma$-algebra.  The set of all
probability measures on the Borel $\sigma$-algebra of $\sD^N$ is denoted by
$\prob^N$. The probabilistic notation $\EE^{\PP}[\cdot]$ is used to denote the
integration with respect to a probability measure in $\prob^N$.

The components of the coordinate process $X$ on $\sD^N$ are generally denoted by
$X^1,\dots, X^N$.  Given a subset  $(X^{i_1},\dots, X^{i_K})$, with $K\leq N$,
of the components of $X$, we denote by $\pi_{X^{i_1},\dots, X^{i_K}}$ the
projection map $\sD^N \to \sD^K$.  For $\PP\in\prob^N$, $\pi_{X^{i_1},\dots,
X^{i_K}}$ induces a probability measure on $\sD^K$, which we call the
$(X^{i_1},\dots, X^{i_K})$-marginal of $\PP$  and denote simply by
$\PP_{X^{i_1},\dots, X^{i_K}}$.

Often, we group sets of variables into single-named vector-valued components to
increase readability. The dimensionality of these components will always be
clear from the context, with the definition of the marginal extending naturally.
To make it easier for the reader, we often employ the notation of the form
$\sD^{d+k}(L,A)$ or $\prob^{d+k}(L,A)$ to signal the fact that the first $d$
coordinates are collectively denoted by $L$ and the remaining $k$ by $A$.  In
the same spirit, we consider (raw) filtrations of the form
$\FFF^Y=\prf{\sF^Y_t}$, $\sF^Y_t = \sigma( Y^1_s, \dots, Y^K_s ; s\leq t)$,
$t\in [0,T]$, on $\sD^N$,   with $Y$ denoting some (or all) components of $X$.
The notation for their right-continuous enlargements is
$\FFF^Y_+=\prf{\sF^Y_{t+}}$, where $\sF^Y_{t+}=\cap_{s>t} \sF^Y_s$. Unless
explicitly stated otherwise,  the usual conditions of right-continuity and
completeness \emph{are not} assumed. When the filtration is, indeed, completed,
and the measure $\PP$ under which the filtration is completed is clear from the
context, we add a bar above $\FFF$
(as in $\bar{\FFF}^Y_{+}$, e.g.).

Some of the components of the coordinate process will naturally come with
further constraints, most often in the form of monotonicity: the subset $\sDm^N$
of $\sD^N$ denotes the class of (component-wise) nondecreasing paths $A$ with
$A_0\geq 0$ (this is natural in our context because we will think of all
functions as taking the value $0$ on $(-\infty,0)$). If monotonicity is required
only for a subset of components, the suggestive notation $
\sD^{K_1+K_2}_{\cdot, \uparrow}$ is used. The
intended meaning is that only the last $K_2$ components are assumed to be
nondecreasing. Similarly, if the monotonicity requirement  is replaced by that of
finite variation, the resulting family is denoted by $\sDf^K$ (unlike in the
case of $\sDm^K$, no nonnegativity requirement on $A_0$ is imposed for
$\sDf^K$). Analogous notation will be used for sets of probability measures, as
well.

For $A\in\sDf^1$ and a measurable (sufficiently integrable) function $f:[0,T]\to
\R$, we use the appropriately-adjusted version of the Stieltjes integral. Namely,
we
define
\begin{align*}
  \int_{[0,T]} f(t)\, dA_t := f(0) A_{0} + \int_0^T f(t)\, dA_{t},
\end{align*}
where the integral on the right is the standard Lebesgue-Stieltjes integral on
$(0,T]$, of $f$ with respect to $A$.  This corresponds to the interpretation of
the process $A$ as having a jump of size $A_0$ just prior to time $0$. This
way, we can incorporate an initial jump in the process $A$ while staying in the
standard \cd{}  framework; the price we are comfortable with paying is that the
implicit value $A_{0-}=0$ has to be fixed. For multidimensional integrators and
integrands, the same conventions will be used, with the usual interpretation of
the multivariate integral as the sum of the component-wise integrals.

\subsection{The monotone-follower problem}

Given $d,k\in\N$, we  consider the path space $\can(L,A)$, where $L$ plays the
role of the target and $A$ the (controlled) monotone follower. As mentioned
above, the natural, raw, $\sigma$-algebras generated by the processes $L$ and
$A$ are denoted by $\FFF^L=\prf{\sF^L_t}$ and $\FFF^A=\prf{\sF^A_t}$,
respectively. A central object in the problem's setup is the probability measure
$\PP_0$ on $\sD^{d}$ which we interpret as the law of the dynamics of the
target.  No additional assumptions are placed on it at this point, but  for
some of
our results to hold, we will need to  require more structure later. On the other
hand,
all our results go through if $L$ is assumed to take values in a Hausdorff
locally-compact topological space with countable base
instead of $\R^d$,
but we keep everything Euclidean for simplicity.

In the spirit of our weak approach, we control the follower by choosing its
joint distribution with the target $L$, in a suitably defined admissibility
class. In the definition below, the condition $\PP_L=\PP_0$ ensures that $L$ has
the prescribed marginal distribution, while the conditional-independence
requirement imposes a form of non-anticipativity on the control:

\begin{definition}[Admissible controls]
  \label{def:adm-contr}
  A probability $\PP\in\canp(L,A)$ is called \define{admissible},
  denoted by $\PP\in\sA$, if \begin{enumerate}
  \item $\PP_L=\PP_0$, and
  \item for each $t\geq 0$, conditionally on $\sF^L_{t+}$, the $\sigma$- algebras $\sF^A_{t}$ and $\FF^L_T$ are $\PP$-independent.
  \end{enumerate}
  If, additionally,
  $\sF^A_t \subseteq \sF^{L}_{t+},\text{ for all }t\in [0,T]$, up to
  $\PP$-negligible sets,
  we say that $\PP$ is \define{strongly admissible}.
\end{definition}

\begin{remark} The condition (2) in Definition \ref{def:adm-contr} above
can be thought of as a non-anticipativity constraint where additional,
$L$-independent, randomization is allowed; it is a version of the so-called
\emph{hypothesis ($\mathcal H$)} of Br\'emaud and Yor (see
\cite{BreYor78}). We point out that the choice of the right-continuous
augmentation $\sF^L_{t+}$ is crucial for our results to hold (see Example
\ref{exa:counter} below), but also that it reverts to the usual hypothesis
$(\mathcal H)$ as soon as a version of the Blumenthal's 0-1 law holds for
$L$.  \end{remark}

The quality of the tracking job  is measured by a
nonnegative convex cost functional:

\begin{definition}[Cost functionals] \label{def:cost-funct} A map
$C:\can(L,A)\to [0,\infty]$, is called a \define{cost   functional} if
there exist measurable functions
\begin{align*}
f:[0,T]\to [0,\infty)^k,\, g:\R^{d}
\times [0,\infty)^k \to [0,\infty) \eand h:\R^{d} \times [0,\infty)^k \to
[0,\infty),
\end{align*}
such that $f$ is continuous,  $h(l,\cdot)$, and $g(l,\cdot)$
are convex on $[0,\infty)^k$, for each $l\in\R^d$, and
\begin{align*}
C(L,A) =
\int_{[0,T]} f(t)\, dA_t + \izT h(L_t,A_t)\, dt + g(L_T,A_T).
\end{align*}
\begin{remark}
\label{rem:5049}
The role of the process $L$ in the cost functional $C$ above is two-fold.
Some of its components play the role of the target to be
tracked, while the others allow the functions $h$ and $g$ to depend on time
or on the randomness from the environment. We enforce this interpretation
in the sequel by making as few assumptions on $L$ as possible, in particular about its
relation to $A$. See Remark, \ref{rem:sufficient}, (2), as well.
\end{remark}
\end{definition}

\begin{definition}[Cost associated with a control]
\label{def:cost-contr}
Given a cost functional $C$ and an admissible probability $\PP\in\sA$, the
\define{(expected) cost} $J(\PP)$ of $\PP$ is given by
\begin{align*}
J(\PP) =
\EE^{\PP}[C(L,A)]\in [0,\infty],
\end{align*}
where $L$ and $A$ denote the components
of dimensions $d$ and $k$, respectively, in $\can$.
\end{definition}

\begin{definition}[Value and solution concepts]
\label{def:mon-fol}
The \define{value} of the monotone-follower problem is given by
\begin{align*}
V = \inf_{\PP\in\sA} J(\PP).
\end{align*}
A probability measure $\hPP\in\sA$ is said to be a
\define{weak solution} to the monotone-follower problem
if $J(\hPP)<\infty$ and  $V=J(\hPP)$. If such $\hPP$ is strongly admissible, we say that the solution is \define{strong}.
For $\eps>0$, a (weak or strong) \define{$\eps$-optimal solution} is a (weakly or strongly) admissible $\PP$ with $J(\PP)<V+\eps$.
\end{definition}

\subsection{ An existence result}
Our first result establishes existence in the monotone-follower problem
(Definition \ref{def:mon-fol})
under weak conditions. Here, and in the sequel, $\abs{\cdot}$ denotes the Euclidean norm on $\R^k$.
\begin{theorem}[Existence under linear coercivity]
\label{thm:exist}
Suppose that the cost function $C$ is linearly coercive, i.e., that
there exist  constants $\kappa, K>0$ such that
\begin{equation}
\label{asm:coercive}
\EE^{\PP}[ C(L,A) ] \geq \kappa \EE^{\PP}[ \abs{A_T} ],
\ \text{ for all }\  \PP\in\sA \text{ with } \EE^{\PP}[ \abs{A_T}]\geq K.
\end{equation}
Then the monotone-follower problem admits a strong solution
whenever its value is finite.
\end{theorem}

\begin{remark}
The reader will immediately notice that the linear coercivity condition
\eqref{asm:coercive} is a fairly weak requirement, guaranteed by either strict
positivity of $f$, or uniform (over $l$) boundedness from below of the function
$g$ by a strictly increasing linear function in $a$, for large $a$. Small modifications of our
results can be made to deal with the case $g=0$, when similar, linear,
coercivity is asked of $h$. Similarly,
one can relax \eqref{asm:coercive} even further by
passing to an equivalent probability measure on the right-hand side.  We leave
details to the reader who comes across a situation in which such an extension is
needed.
\end{remark}

The following two examples show that neither one of the two major
conditions - linear coercivity of \eqref{asm:coercive} in Theorem
\ref{thm:exist}, or the use of the right-continuous augmentation
$\sF^L_{t+}$ in the definition of admissibility (Definition
\ref{def:adm-contr}) - can be significantly relaxed:
\begin{example}[Necessity of assumptions] \label{exa:counter} As for the
  coercivity assumption \eqref{asm:coercive}, a trivial example can be
  constructed with $T=1$, $f=0$, $h=0$, $k=d=1$, $g(l,a)=e^{-a}$ and an
  arbitrary $\PP_0$. The value of the problem is clearly $0$, but no
  minimizer exists. Linear coercivity clearly fails, too.

  In order to argue that the right-continuous augmentation $\FFF^L_+$ in the
  Definition \ref{def:adm-contr} is necessary, we take  $T=1$ and assume that
  the dynamics of the target satisfies
  \begin{align*}
    L_t=t L_1 \efor t\in [0,1],
    \PP_0\mathrm{-} \as, \text{ with } \PP_0[L_1=1]=\PP_0[L_1=0]=\tot,
  \end{align*}
  and
  that the cost functional is given by
  \begin{align*}
    C(L,A)=\int_{[0,1]}  (\tfrac 1
    2+t)\,dA_t+\int_0^1 \abs{L_t-A_t}\,dt.
  \end{align*}
  Let $\PP^*\in \canp(L,A)$ be such
  that $\PP^*_L=\PP_0$ and $A^*_t = \oo4 L_1$, for all $t\in [0,1]$,
  $\PP^*$-a.s. Since the admissibility requires that $\sigma(A_0)$ and
  $\sigma(L_T)$ be independent, $\PP^*$ is not admissible. It does have the
  property that
  \begin{equation}
    \label{equ:16EF}
    \begin{split}
      J(\PP^*)\leq J(\PP),\text{ for each }\PP\in\canp\text{ with }\PP_L=\PP_0.
    \end{split}
  \end{equation}
  Indeed, one can check that
  \begin{align*}
    C(0,0) \leq C(0,\alpha) \eand C(\iota, \oo4 \iota)\leq C(\iota, \alpha) \text{
    for all $\alpha\in \sDm^1$},
  \end{align*}
  where $\iota$ and $0$ denote the identity and the constant $0$ function
  on $[0,1]$, respectively.
  Moreover, the inequality in \eqref{equ:16EF} is an equality if and only
  if $\PP=\PP^*$. Thus, to show that no admissible minimizer exists it will be
  enough to find a sequence $\seq{\PP}$ in $\sA$ such that $J(\PP_n)
  \downto J(\PP^*)$. This can be achieved easily by using the $\PP_0$-laws of
  $(A^n,L)$, where
  \begin{align*}
    A^n_t = \begin{cases} \tot, & t< \tfrac{1}{n},\\  \oo4 L_t , & t \in
    [\oo{n},1], \end{cases}
  \end{align*}
\end{example}

\subsection{A characterization result}
Using the same ingredients
as in the formulation of the monotone-follower problem,
we pose a
forward-backward-type stochastic equation (called the \define{Pontryagin FBSDE}),
as a formulation of the maximum principle of Pontryagin.
Whenever the Pontryagin FBSDE is involved, we automatically assume that both $a\mapsto h(l,a)$ and $a\mapsto g(l,a)$ are continuously differentiable in $a$ on $[0,\infty)^k$
for each $l$, and denote their gradients (in $a$) by $\nabla h $ and $\nabla g $, respectively. Any inequalities between multidimensional processes are to be understood componentwise.
\begin{definition}[The Pontryagin FBSDE]
  \label{def:FBSDE}
  A probability  $\tPP\in
  \prob^{d+2k}(L, A,Y)$
  is said to be a \define{weak solution} of the Pontryagin
  FBSDE if
  \begin{enumerate}[leftmargin=1.7em]
    \item $\tPP_{L,A}\in\sA$,
    \item $Y\geq 0$ and $\izT Y_t\, dA_t=0$, $\tPP$-a.s.
    \item $Y+\int_0^{\cdot} \nabla h (L_t,A_t)\, dt - f$ is an
      $(\FFF^{L,A,Y},\tPP)$-martingale with
      $Y_T = f(T)+\nabla g (L_T,A_T)$, $\tPP$-a.s.
  \end{enumerate}
\end{definition}
\begin{remark} Under $\tPP$ as above, $(L,A,Y)$ can be interpreted as a (weak)
solution to a fully-coupled stochastic forward-backward differential equation
with reflection. Indeed, the forward component $(L,A)$ feeds into the backward
component $Y$ directly (and through the terminal condition). On the other hand,
the backward component affects the forward component through the reflection term
in Definition \ref{def:FBSDE}, (2). The usual stochastic-representation
parameter $Z$ is hidden in our formulation (in the martingale property of $Y$ as
we do not assume the predictable-representation property in any form) and it
does not feed directly into the dynamics. For that reason, it would perhaps be
more appropriate to call (1)-(3) above a forward-backward stochastic equation
(FBSE) instead of FBSDE; we choose to stick to the canonical nomenclature,
nevertheless. \end{remark}

The main significance of the Pontryagin  FBSDE
lies in the following characterization:
\begin{theorem}[Characterization via the Pontryagin FBSDE]
\label{thm:max-equiv}
Suppose that the functions $g(l,\cdot)$ and $h(l,\cdot)$ are
convex and continuously differentiable  on $[0,\infty)^{k}$ for each $l\in\R^d$.
\begin{enumerate}
\item Suppose that there exist  Borel functions $\Phi_{g}, \Phi_h:\R^d \to [0,\infty)$ and
a  constant $M\geq 0$, such that
\begin{align*}
\int_0^T \Phi_h(L_t)\, dt + \Phi_g(L_T) \in \lone(\PP_L),
\end{align*}
and, for $\vp\in\set{g,h}$,
\begin{align*}
\abs{\nabla \vp(l,a)} & \leq \Phi_{\vp}(l)+ M\vp(l,a),\eforall (l,a)\in\R^{k}\times [0,\infty).
\end{align*}
Then each solution $\hPP$ of the monotone follower problem is an $( L,
A)$-marginal of some solution $\tPP$ to the Pontryagin FBSDE.
\item   If the Pontryagin FBSDE admits a solution $\tPP$, then its   marginal
$\tPP_{ L, A}$ is a solution of the monotone-follower problem whenever its
value is finite
\end{enumerate}
\end{theorem}
\begin{remark}\label{rem:exponential}\ 
\begin{enumerate}
\item Our Pontryagin FBSDE can be interpreted as a weakly-formulated version of 
(stochastic) first-order conditions. These can be found in the literature, in
settings similar to ours, and in the context of singular control, e.g., in \cite{BanRie01a},
         \cite{Ban05}, 
         or, more recently, 
         \cite{Ste12}).
\item
The condition in (1) above essentially states that $\vp$ grows no faster than an
exponential function, with the parameter uniformly bounded from above in $l$.
This should be compared to virtually no growth condition needed for
existence in Theorem \ref{thm:exist}, as well as to the polynomial growth
conditions needed for the approximation result in Theorem \ref{thm:approx}
below. 
\end{enumerate}
\end {remark}
While we will be using the Pontryagin FBSDE mostly as a tool in the proof of Theorem \ref{thm:approx}, we believe that the the following result, which is an immediate consequence of Theorems \ref{thm:exist} and \ref{thm:max-equiv} above merits to be mentioned in its own right. \begin{corollary}[Existence for the Pontryagin FBSDE]
  \label{cor:FBSDE}
  Under the combined assumptions of Theorems
  \ref{thm:exist} and \ref{thm:max-equiv}, part (1),
  the Pontryagin FBSDE admits a solution, as soon as the value of the monotone-follower problem is finite.
\end{corollary}
\begin{remark}\label{rem:uniqueness} We do not discuss uniqueness of solutions
in detail
either in the context of Theorem \ref{thm:exist} above, or in the context of our
other results below. In particular cases, clearly, the strong solution will be
unique if
enough strict convexity is assumed on the problem ingredients.
\end{remark}
\subsection{A connection with optimal stoppeng}\
\label{sub:os}
In our next result, we revisit, and, more importantly, reinterpret, the
celebrated connection between optimal stopping and stochastic control in
the context of the generalized monotone-follower problem 
in
dimension $k=1$. Our formulation of the optimal-stopping problem differs
slightly from the classical one, but is easily seen to be essentially
equivalent to it (we comment more about it below). It is chosen so as to
make our point - namely that the stopping problem associated to the
monotone-follower problem is but a manifestation of the maximum principle
of Pontryagin - more prominent.  It also follows our distributional
philosophy and we get to  reuse the framework (and the notion) of
admissible controls $\sA$ from Definition \ref{def:adm-contr}. 

Specifically, we
work on the path space $\cano(L,A)$ 
and,  assuming that the functions $g$
and $h$ are continuously-differentiable in $a$, with derivatives
denoted by $g_a$ and $h_a$, we define
  \begin{equation}
    \label{equ:stop-prob}
    \begin{split}
      K(\PP) = \EE^{\PP}\left[
      \Big(f(\tau_A) +  g_a (L_{\tau_A}, 0)+
      \int_{\tau_A}^{T} h_a (L_t,0)\, dt\Big)
      \inds{\tau_A<\infty}\right] \in \R,
    \end{split}
  \end{equation}
  where $\tau_A$ is the stopping time given
  by \begin{align*}
  \tau_A = \inf\sets{t\geq 0}{ A_t>0},  
  \text{ with } \inf\emptyset =
  +\infty,
  \end{align*}
  whenever 
  the expression
  inside the expectation in \eqref{equ:stop-prob}  above
  is in $\lone(\PP)$; the set of all such $\PP$ is denoted by $\sA_S$.
\begin{definition}
  A probability $\hPP\in\sA_S$
  is said to be  
   a \define{solution} of the optimal-stopping
  problem if $K(\hPP) \leq K(\PP)$ for all $\PP\in\sA_S$.
\end{definition}

\begin{remark}\
\label{rem:4C83}
Viewed in isolation, the above formulation of the optimal stopping problem
contains obvious redundancies (the $\PP$-behavior of $A$ after $\tau_A$,
for example). Even when  the class of the probability measures $\PP\in\sA$ is
further restricted so that $A$ becomes a single-jump $0$-to-$1$ process,
$\PP$-a.s., our formulation corresponds to a randomized
optimal stopping problem, in that $A$ is allowed to depend on innovations
independent of $L$. All in all, part (2) of Definition \ref{def:adm-contr}
makes the problem equivalent to a randomized  optimal stopping problem
with respect to the right-continuous augmentation of $\prf{\sF^Y_t}$.
There is no harm, however, since it  turns out
that, as usual in optimal stopping, randomization leads to no increase in value.
\end{remark}
\begin{theorem}[A connection between control and optimal stopping]
\label{thm:opst}
Suppose that $k=1$ and that
the assumptions
of
Theorem \ref{thm:max-equiv}, part (1), hold. Then
any solution
to the monotone-follower problem
is also a  solution to the optimal-stopping
problem. 
\end{theorem}
\begin{remark}
As we do not use the notion of a value function,
there is no analogue of the equation (3.17) in Theorem 3.4, p.~862 in
\cite{KarShr84} about equality between the derivative (gradient) of the
value function in the control problem and the value of the optimal stopping
problem. The statements about the relationship between the optimal control
in the former and the optimal stopping time in the later translate directly
into our setting. The reader will see that the (short) proof of Theorem
\ref{thm:opst} below, given in subsection
\ref{sse:prstop}, it is nothing but a simple observation, once the
Pontryagin principle is established.
\end{remark}
\subsection{The approximation result}
\label{sub:cap}
In order to understand the monotone-follower problem better and to provide
an approach to it with computation in mind,
we pose a sequence of its ``capped'' versions. These play the role of natural regular approximands to the inherently singular monotone-follower problem.
The setting follows closely that of the
previous section. The only difference is that the set of allowed controls
consists only of Lipschitz-continuous nondecreasing processes, without the initial jump.
More precisely, we have the following definition:
\begin{definition}[Admissible capped controls]
  \label{def:adm-capped-contr}
  Given $n\in\N$,
  a probability $\PP\in\canp(L,A)$ is called \define{$n$-capped admissible},
  denoted by $\PP\in\sAwn$, if $\PP\in\sA$ and, $\PP$-a.s., the coordinate
  process $A$ is Lipschitz continuous with the Lipschitz constant at most
  $n$, and $A_0=0$, $\PP$-a.s.
  The \define{value} of the $n$-th capped
  problem is given by
  \begin{align*}
    \Vwn = \inf_{\PP\in\sAwn} J(\PP),
  \end{align*}
  and we say that the probability measure $\hPP\in\sAwn$ is the \define{weak
  solution} to the capped monotone-follower problem  if $\Vwn=J(\hPP)<\infty$.
\end{definition}
While Theorem \ref{thm:exist} relied on a minimal set of assumptions, the
approximation result we give below requires more structure. Here, $C_c^{
\infty} (\R^d)$
denote the set of all infinitely-differentiable functions on $\R^d$ with
compact support,
while $C_b(\R^d)$ refers to the set of all bounded continuous functions; $\ld$
denotes the Lebesgue measure on $[0,T]$.
\begin{theorem}[Approximation by regular controls]
\label{thm:approx}
Suppose that
\begin{enumerate}
  \item The law $\PP_0$ is Feller, in that for each $t\in [0,T)$
    \begin{enumerate}
      \item the $\sigma$-algebras $\sF^L_{t+}$ and $\sF^L_t$ on $\sD^d$ coincide $\PP_0$-a.s.
      \item for each $G\in C^{\infty}_c(\R^d)$, there  exists $G^*\in C_b(\R^d)$
      such that
        \begin{align*}
          \EE^{\PP_0}[ G(L_T)| \sF^L_{t+}] = G^*(L_t),\ \PP_0\text{-a.s.}
        \end{align*}
    \end{enumerate}
  \item The coordinate process $L$ is a quasimartingale under $\PP_0$
  \item The primitives $f,g$ and $h$ are regular enough, in that
    \begin{enumerate}
      \item
        each component of $f$ is uniformly bounded away from $0$,
       \item  the functions $g(\cdot,a)$ and $h(\cdot,a)$ are
        continuous for each $a\in [0,\infty)^k$.
      \item
        $h(l,\cdot)$,  and
        $g(l,\cdot)$ are continuously differentiable and convex on $[0,\infty)^{k}$ for each $l\in\R^d$, and
        there exist $p,q>1$
        and Borel functions $\Phi_g, \Phi_h:\R^d \to [0,\infty)$ with
        \begin{align*}
          \Phi_h(L) \in \lpee(\ld \otimes \PP_0 ) \eand
          \Phi_g(L_T) \in \lpee(\PP_0),
        \end{align*}
        such that, for $\vp\in\set{g,h}$, we have
        \begin{align*}
          \vp(l,0)+\abs{\nabla \vp (l,a)}  \leq \Phi_{\vp}(l) + \abs{a}^q,\text{ for all } (l,a)\in\R^{d+k}.
        \end{align*}
    \end{enumerate}
\end{enumerate}
Then
\begin{itemize}
  \item For each $n$, the capped problem
    admits a solution $\hPPn\in \sAwn$ and
    \begin{align*}
      \Vwn \downto V.
    \end{align*}
  \item
    A subsequence of the sequence $\sq{\hPPn}$
    converges in the Meyer-Zheng sense to a solution $\hPP$ of the monotone
    follower problem.
\end{itemize}
\end{theorem}
\begin{remark}\label{rem:sufficient}\
\begin{enumerate}
\item
There are several slightly-different classes of processes found under the name
of a Feller process in the literature, so we make the essential properties
needed in the proof explicit in the statement. These particular properties are,
furthermore, implied by
all the definitions of the Feller property the authors have encountered.
Consequently,
all standard examples of Feller processes such as diffusions,
stable processes, L\' evy processes, etc., fall under our framework.
\item The quasimartingality assumption on $L$ is put in place mostly for
convenience. It is known that so-called ``nice'' Feller processes (the
domain of whose generator contains smooth functions with compact support) are
automatically special semimartingales and, therefore, local quasimartingales
(see \cite{Sch12} for the first part of the statement, and 
 \cite[Theorem 23.20, p.~451]{Kal02b} for the second). 
As no convexity in the variable $l$ is assumed, one can further do away with
the
localization in many cases by replacing $L$  by $q(L)$, where $q$ is
a smooth, injective and bounded function. Such a replacement would not
change the problem; indeed, conditions (1) and (3) of Theorem \ref{thm:approx}
are invariant under the transformation $L \mapsto q(L)$.
\item The growth assumptions on the functions $f$, $g$ and $h$ are essentially
those of \cite{KarShr84}, rephrased in our language. We note the fact that
$f$ is bounded away from zero immediately implies the linear coercivity
condition of Theorem \ref{thm:exist}, while the condition $\vp(l,0)\leq \Phi_
{\vp}(l)$, for $\vp\in \set{g,h}$, guarantees that the value is finite.
\end{enumerate}\end{remark}
\begin{example}
In general, the sequence of capped optimizers cannot be guaranteed to converge
towards a minimizer $\PP^*$ weakly, under the the Skorokhod topology. Indeed,
Skorokhod
convergence preserves continuity, and all capped optimal controls are
continuous, but it is easily seen that the solution to the
monotone-follower problem does not need to be a continuous process. Indeed,
it suffices to take $k=d=1$,  any $\PP_0$ with $\PP_0[L_T>1]>0$,
$f\equiv 1$, $h\equiv 0$ and  $g (l,a) = \tot (l-a)^2$, so
that the optimal $A$
is given by $A_t=0$ for $t<T$ and $A_T= \max(0,L_T-1)$.

On the other hand, if one can guarantee that the optimizer is continuous (and
$A_0=0$), the Meyer-Zheng convergence automatically upgrades to the weak
convergence in $C[0,T]$ (see \cite{Pra99}).
\end{example}
One of the immediate consequences of Theorem \ref {thm:approx}
is that the monotone-follower problem can be posed over Lipschitz controls, without affecting the value function.
\begin{corollary}[Lipschitz $\eps$-optimal controls]
  \label{cor:relax} Under the conditions of Theorem \ref{thm:approx}
  for each $\eps>0$ there exists $M>0$ and an
  $\eps$-optimal admissible control $\PP$, such that $A$ is uniformly
  $M$-Lipschitz, $\PP$-a.s.
\end{corollary}
\section{Proofs}
Proofs of our main results, namely Theorems \ref{thm:exist},
\ref{thm:max-equiv}, \ref{thm:opst} and \ref{thm:approx}  are collected in
this section.
The proof of each theorem
occupies a section of its own, and all the conditions stated in the theorem
are assumed to hold - without explicit mention - throughout the section.
\subsection{A proof of Theorem \ref{thm:exist}}
We start with an
auxiliary result which states that an admissible control
can always be turned into
a strong admissible control without any sacrifice in value.{ The central idea is
that, even though the optional projection of a nondecreasing process is
\emph{not necessarily
nondecreasing} in general, this turns out to be so in our setting. 
\begin{proposition}
  \label{pro:optional}
  For $\PP\in\sA$ with $\EE^{\PP}[A_T]<\infty$ let $\oA$ be the optional projection of $A$ onto the right-
  continuous and complete augmentation $\bar{\FFF}^L_+$
  of the natural filtration $\FFF^L$. Then the joint law $\oPP$ of $(L,\oA)$ is
  admissible and $J(\oPP)\leq J(\PP)$.
\end{proposition}
\begin{proof}
  The optional projection of a \cd{}
  process onto a filtration satisfying the usual conditions is indistinguishable from a
  \cd{} process
  (see, e.g.,  Theorem 2.9, p.~18 in \cite{BaiCri09}).
  It is an immediate consequence of the condition
  (2) of Definition \ref{def:adm-contr} that
  \begin{align*}
    \EE^{\PP}[ A_t |\bFl_{t+}] = \EE^{\PP}[ A_t| \sF^L_T],
    \text{a.s., for all } t\in [0,T],
  \end{align*}
  and, so $\oA_t=\EE^{\PP}[A_t|\sF^L_T] \leq
  \EE^{\PP}[ A_s|\sF^L_T]=\oA_s$, a.s., for $s\leq t$.
  By construction, the $\sigma$-algebras $\bFl_{t+}$ and $\sF^L_{t+}$ differ
  only in $\PP$-negligible sets, and, so, $\bar{\sF}_T$ and $\bFl_{t+}$ are
  conditionally independent given $\sF^L_{t+}$, which, in turn,  implies  that the
  joint law of $(L,A)$ is admissible.

  Next, we show that  $J(\oA)\leq J(A)$.
    For $\vp \in \set{g,h}$ we denote by $\tilde{\vp}(l,\cdot)$ the convex
      conjugate (in the second variable) of $\vp$:
      \[ 
      \tilde{\vp}(l,\alpha) = \sup_{a\geq 0} \Big( \alpha a - \vp(l,a)
          \Big) \text{ so that }
  \vp(l,a) = \sup_{\alpha \in \R} \Big( \alpha a -
      \tilde{\vp}(l,\alpha)
      \Big).  \] Then, for any bounded $\sF^L_T$-measurable random variable
    $\alpha_T$ with $\tilde{\vp}(L_t, \alpha_T)<\infty$, $\PP$-a.s., we have
    \[ 
    \EE^{\PP}[ \vp(L_t,A_t) | \sF^L_T] \geq
    \EE^{\PP}[ \alpha_T A_t|\sF^L_T]
    - \tilde{\vp}(L_t,\alpha_T)  = \alpha_T \oA_t - \tilde{\vp}(L_t,\alpha_T),
    \text{ $\PP$-a.s.}\]
      The $\PP$-essential supremum of the right-hand side over all bounded
      $\sF^L_T$-measurable $\alpha_T$ is easily seen to be equal to
      $\vp(L_t,\oA_t)$, $\PP$-a.s., for
      $t\in [0,T]$, so, by the tower property, 
    $\EE^{\PP}[
      \vp(L_t, \oA_t) ] 
        \leq  \EE^{\PP}[ \vp(L_t,A_t)]$.
  Thus,
    \begin{align*} 
    \EE^{\PP}[ \int_0^T h(L_t,\oA_t)\, dt
      + g(L_T,\oA_T)] \leq \EE^{\PP}[ \int_0^T h(L_t,A_t)\, dt + g(L_T,A_T)].
      \end{align*} 
  
Finally, we let $\sM$ denote the set of all bounded
measurable functions $\psi:[0,T]\to \R$ with
\begin{equation}
  \label{equ:mcl}
    \EE[ \int_{[0,T]} \psi(t)\, d\oA_t]=\EE[\int_{[0,T]} \psi(t)\, dA_t].
\end{equation}
$\sM$ is clearly a monotone class which contains all functions of the form
$\psi(t) = \ind{(a,T]}(t)$, so, by the monotone-class theorem, it contains all bounded measurable functions and, in particular, $f$.
\end{proof}
Continuing with the proof of Theorem \ref{thm:exist}, we assume
that its value  is finite, pick a minimizing sequence $\seq{\PP}\subseteq\sA$,
and use it to build a probability space $(\Omega,\FF,\PP)$ and, on
it,  the sequence $L$, $A^{(1)}$, $A^{(2)}$, \dots, as in Lemma
\ref{lem:prob-space}.

Thanks to Proposition \ref{pro:optional}, we may assume, without loss of generality,
that all $\An$ are $\bar{\FFF}^L_+$-adapted, where
$\bar{\FFF}^{L}_+=\prf{\bFl_{t+}}$ denotes the
right-continuous and complete augmentation of the natural filtration $\prf{\sF^L_t}$,

Now that a common probability space has been constructed, 
we follow the
methodology of \cite{BanRie01a} and \cite{RieSu11}. 
Thanks to the linear coercivity condition \eqref{asm:coercive}, the sequence
$\sq{\An_T}$ is bounded in $\lone$;
also, all $\An$ are $\bar{\FFF}^L_+$-adapted, and $\bar{\FFF}^L_+$
is right-continuous.
Therefore, we can use 
Lemma 3.5, p.~470, in \cite{Kab99} to guarantee the existence of 
an $\bar{\FFF}^L_+$-adapted process
$B$, with paths in $\sD^d$ and  
a sequence $\sq{\Bn}$ of Ces\` aro means of a subsequence of
$\seq{\An}$ which converges to $B$ in the following sense (the sense of
    optional random measures):  
  for almost all $\omega$, the Stieltjes measures induced by
  $\Bn(\omega)$ converge weakly towards to the Stieltjes measure induced by
  $B(\omega)$. In particular, there exists a countable subset 
  $\sN$ of $[0,T)$ (the set of jumps of $t\mapsto \EE[B_t]$ on
    $[0,T)$) such that 
  \[ \int_0^T f(t)\, d\Bn_t \to \int_0^T f(t)\, dB_t,\text{ a.s.,}\ \text{ and \ }
\Bn_t \to B_t, \text{ a.s., for all $t\in [0,T]\setminus \sN$}.\]
  Therefore, 
  by Fatou's lemma (applied on $\Omega$ for the first and the third term, 
      and on the product space
      $[0,T]\times \Omega$ for the second), we have
    \begin{multline*}
\EE[ \int_0^T f(t)\, dB_t+\int_0^T h(L_t,B_t)\, dt + g(L_T,B_T)] \leq \\
  \leq \liminf_{n\to\infty}
\EE[ \int_0^T f(t)\, d\Bn_t+ \int_0^T h(L_t,\Bn_t)\, dt + g(L_T,\Bn_T)].
  \end{multline*}
For a nondecreasing \cd{} process $A$ on $(\Omega,\FF,\PP)$ we set $J(A) =
  \EE[ C(L,A)]$ and notice that the convexity of $J$ and the fact that
  $J(\An)\downto \inf_{\PP\in\sA} J(\PP)$  together yield
  that $\sq{\Bn}$ is a minimizing sequence, too, in that 
  $J(\Bn) \downto \inf_{\PP\in\sA} J(\PP)$. Therefore, $J(B) \leq
  \inf_{\PP\in\sA} J(\PP)$ and 
  it only remains to note that
  the law of $(L,B)$ is
  strongly admissible
  since 
  $B$ is $\bar{\FFF}^L_+$-adapted.

\subsection{A proof of Theorem \ref{thm:max-equiv}}
To streamline the presentation in this and the subsequent subsections,
we introduce additional notation:
the \define{subgradient} map $\da C(L,A): [0,T] \to \R^k$, at $(L,A) \in\can$,
is given  by
\begin{align*}
  \da C(L,A)_t =  f(t) + \int_{t}^{T} \nabla h (L_s,A_s)\, ds +
  \nabla g (L_T,A_T) \efor t\in [0,T],
\end{align*}
where, as usual,  $\nabla h $ and $\nabla g $ denote the gradients with respect to the
second variable. The reader will easily check that $\da C(L,A)$ has the
following property (which earns it the name subgradient):
\begin{equation}
  \label{equ:sbcx}
  \begin{split}
    C(L,A+\Delta) \geq  C(L,A) + \scl{\da C(L,A)}{\Delta},
  \end{split}
\end{equation}
for all $\Delta\in\sDf^k$ with $A+\Delta\in \sDm^k$, where
\begin{align*}
  \scl{X}{\Delta}=
  \int_{[0,T]} X_u\, d\Delta_u.
\end{align*}
We also note, for future reference and using integration by parts, that
\begin{align}
  \label{equ:intpart}
  \scl{\da C(L,A)}{ \Delta} = \int_{[0,T]} f(t)\, d\Delta_t
  + \int_0^T \nabla h (L_t,A_t) \Delta_t\, dt + \nabla g (L_T,A_T) \Delta_T,
\end{align}
for all $\Delta \in\sDf^k$.

We start the proof by assuming that $\hPP\in\prob^{d+k}(L,A)$ solves the
mo\-no\-to\-ne-follower problem, with value $V=J(\hPP)<\infty$. In particular, we have $C(L,A)\in\lone(\hPP)$.  To relieve
the notation we work on the sample space $\Omega=\can(L,A)$, under the
probability $\hPP$, until the end of this part of the proof.
Moreover, thanks to assumptions of the theorem,
for $\vp\in\set{g,h}$,
$l\in\R^k$, $a\in [0,\infty)^d$ and $x\in\R^d$ such that $a+x\in [0,\infty)^d$, we have, for each $c\in (0,1)$,
\begin{align*}
  \abs{\nabla \vp(l,a+c x)}&\leq \Phi_\vp(l)+ M \vp(l,a+ cx)\\ &=
  \Phi_\vp(l) + M \vp(l,a)+M \int_0^c \scl{x}{\nabla \vp(l, a+tx)}\, dt\\
  &\leq \Big( \Phi_\vp(l) + M \vp(l,a) \Big) + M \abs{x} \int_0^c \abs{\nabla \vp(l,a+tx)}\, dt
\end{align*}
Gronwall's inequality then implies that
\begin{align}\label{equ:gron}
  \abs{\nabla \vp(l, a+ x)} \leq \Big( \Phi_\vp(l)+ \vp(l,a) \Big) e^{M \abs{x}}.
\end{align}
Let $\sV_A$ denote the set of all bounded processes $\Delta$
with paths in
$\sDf^k$, adapted to the natural filtration $\FFF^{L,A}$
such that,
\begin{align*}
  \text{ either $\Delta\in \sDm^k$ or $\Delta = -\tot\min(A,n)$ for some
  $n\in\N$.}
\end{align*}
It has the property that for $\eps\in [0,1]$ and $\Delta\in\sV_A$, the joint law $\PP^{\eps}$ of $(L,\Ae)$,
where $\Ae = A+ \eps \Delta$,  is an
admissible probability measure in $\prob^{d+k}$. By the optimality of $A$ and \eqref{equ:sbcx}, we have
\begin{align*}
  \EE[ C(L,A)] \leq \EE[ C(L,\Ae)] \leq \EE[ C(L,A) +
  \scl{\da C(L,\Ae)}{\eps \Delta}],
\end{align*}
from where it follows that
\begin{align}\label{equ:variational}
  \scl{\da C(L,\Ae)}{\Delta})^-\in\lone \eand
  \EE[ \scl{\da C(L,\Ae)}{\Delta}] \geq 0,\text{
  for all $\eps \in [0,1]$. }
\end{align}
Thanks to boundedness of processes in $\sV_A$ and the fact that $C(L,A)$ is integrable,
the inequality \eqref{equ:gron} implies that the family
\begin{align*}
  \Bsets{ \scl{\da C(L_t, \Ae_t)}{\Delta}}{ \eps \in [0,1]} \text{
  is uniformly integrable for all }\Delta\in\sV_A.
\end{align*}
Moreover,
both $\nabla h $ and $\nabla g $ are continuous, so
\begin{align*}
  \lim_{\eps \to 0} \scl{\da C(L,\Ae)}{\Delta} = \scl{\da C(L,A)}{\Delta},\as
\end{align*}
It follows that we can pass to the limit as $\eps\to 0$ in \eqref{equ:variational} to conclude that
\begin{equation}
  \label{equ:1A5B}
  \begin{split}
    \EE[ \scl{ \da C(L,A)}{ \Delta}] \geq 0, \text{ for all } \Delta\in\sV_A,
  \end{split}
\end{equation}
and, consequently, that
\begin{align}
  \label{equ:Y-Del}
  \EE[ \scl{ Y}{\Delta}] \geq 0, \text{ for all } \Delta\in\sV_A,
\end{align}
where
$Y$ denotes the optional
projection of $\da C(L,A)$ onto the
right-continuous and complete augmentation $\bar{\FFF}^{L,A}_+$
of $\FFF^{L,A}$.
Since $\da C(L,A)$ is \cd{}, the process $Y$ can be
chosen in a \cd{} version, too (see
Theorem 2.9, p.~18~in \cite{BaiCri09}). Hence,
by varying $\Delta$ in the class of nondecreasing processes in $\sV_A$,
we can conclude that
$Y_t\geq 0$, for all $t\in [0,T]$, a.s.

On the other hand if we use each element of the
sequence $\Delta_n = - \tot \min(A,n)$ in \eqref{equ:Y-Del}, we obtain
\begin{align*}
  \int_{[0,T]} Y_t\, dA_t = 0, \as.
\end{align*}
In order to show that
the law $\tPP$  of the triple $(L,A,Y)$
solves the Pontryagin FBSDE, we only need to argue that $Y+\int_0^{\cdot} \nabla h (L_t,A_t)\, dt - f$ is an $\FFF^{L,A,Y}$ martingale
(under $\tPP$, on $\sD^{d+2k}$). This follows directly from the fact that $Y$ is a \cd{} version of the optional projection of $\da C(L,A)$ onto $\FFF^{L,A,Y}$.

\bigskip

Conversely, let $\tPP\in\prob^{d+2k}( L, A, Y)$ be a
solution to the Pontryagin FBSDE.  To prove that $\hPP=\tPP_{ L, A}$ is a
weak minimizer in the monotone-follower problem, we pick a competing
admissible measure $\PP' \in \sA$.
Using Lemma \ref{lem:coupling},
we construct the measure $\PP=\tPP\otimes \PP'$ on
$\sD^{d+3k}$ (with coordinates $(L,A,Y,A')$). Since
$\PP_{(L,A,Y)}$ solves the Pontryagin FBSDE,
$Y+\int_0^{\cdot} \nabla h (L_t,A_t)\, dt - f$ is an
$(\FFF^{L,A,Y},\PP)$-martingale. Moreover, the $L$-conditional independence between $A'$ and $(A,Y)$ implies that it is also an
$(\FFF^{L,A,Y,A'},\PP)$-martingale. Consequently, we have
\begin{align*}
  \eeq{\PP}{ \scl{\da C(L,A)}{A'}}=
  \eeq{\PP}{ \scl{Y}{A'}}
  \eand
  \eeq{\PP}{ \scl{\da C(L,A)}{A}}=
  \eeq{\PP}{ \scl{Y}{A}}.
\end{align*}
The subgradient identity \eqref{equ:sbcx} then implies that
\begin{equation}
  \label{equ:3A1B}
  \begin{split}
    J(\PP') &=\eeq{\PP}{C(L,A')}\geq
    \eeq{\PP}{ C(L,A) + \scl{\da C(L,A)}{ A'-A}} \\
    &= J(\hPP) +\eeq{\PP}{\scl{Y}{A'-A}}  = J(\hPP) + \eeq{\PP}{\scl{Y}{A'}}\geq J(\hPP).
  \end{split}
\end{equation}

\subsection{A proof of Theorem \ref{thm:opst}} \label{sse:prstop}
Let $\hPP\in\sA$ be a solution to the monotone-follower problem. By Theorem
\ref{thm:max-equiv}, part (1), it can be realized as the marginal $\tPP_{L,A}$ of some solution $\tPP_{L,A,Y}$ of
the Pontryagin FBSDE.
For an
admissible measure $\PP' \in \sA$, and
using Lemma \ref{lem:coupling},
we can construct the measure $\PP=\tPP\otimes \PP'$ on
$\sD^{d+3}_{\cdot \uparrow \cdot \uparrow}$ (with coordinates $(L,A,Y,A')$)
and work on $\sD^{d+3}_{\cdot \uparrow \cdot \uparrow}$ under $\PP$ for the remainder of the proof.
As argued in the previous subsection, the process
$Y+\int_0^{\cdot}  h_a (L_t,A_t)\, dt - f$ 
is an
$(\FFF^{L,A,Y,A'},\PP)$-martingale, and, so,
\begin{align*}
  \EE[Y_{\tau_{A'}}
  \inds{\tau_{A'}<\infty}]
  =
  \ee{ \da C(L,A)_{\tau_{A'}} \inds{\tau_{A'}<\infty}}
\end{align*}
where $\tau_{A'}=\inf\sets{t\geq 0}{A'_t>0}\in [0,T]\cup\set{\infty}$.
By the assumptions of convexity 
we placed on
$h $ and $g$, we have the  following
inequalities 
\begin{align*}
   h_a ( L_s,0)-  h_a ( L_s,A_s)\leq 0 \eand
   g_a (L_T,0)-  g_a (L_T,A_T)\leq 0,
\end{align*}
for all $s\in [0,T]$, a.s.
Therefore, by the nonnegativity of $Y$, we have
\begin{align*}
  K(\PP') & =
  \EE[ \da C(L,0)_{\tau_{A'}} \inds{\tau_{A'}<\infty}]
  \geq  \EE[ \da C(L,0)_{\tau_{A'}} \inds{\tau_{A'}<\infty} - Y_{\tau_{A'}}
  \inds{\tau_{A'}<\infty}]\\
  &=
  \EE\Big[ \int_{\tau_{A'}}^T \big( h_a (L_s,0) -
   h_a (L_s,A_s) \big)\, ds+
  \Big( g_a (L_T,0) -  g_a (L_T,A_T)\Big)\inds{\tau_{A'}<\infty}\Big]\\
  &\geq \ee{ \int_{0}^T \big( h_a (L_s,0) -
   h_a (L_s,A_s) \big)\, ds +
  \Big(g_a (L_T,0) -  g_a (L_T,A_T)\Big) }\\
  &=
  \EE[ \da C(L,0)_0 - Y_0]
\end{align*}
On the other hand, if we repeat the computation above with $\tau_{A'}$
replaced by $\tau_{A}$, all the inequalities become equalities, implying
that $K(\PP) \leq K(\PP')$. Indeed, we clearly have
\begin{align*}
   h_a (L_s,0) =  h_a (L_s,A_s), \text{ on } \set{s<\tau_{A}},
\end{align*}
and
\begin{align*}
   g_a (L_T,0) =  g_a (L_T,A_T)\text{ on } \set{\tau_A=\infty},
\end{align*}
as well as
\begin{align*}
  \EE[ Y_{\tau_A} \inds{\tau_A<\infty}] =0,
\end{align*}
where this last equality follows from the fact that $\int_0^T Y_u\, dA_u=0$.
\subsection{A proof of Theorem \ref{thm:approx}}
\label{sse:proof-approx}
We start by posing the capped monotone-follower problems on a common fixed probability
space $(\Omega,\FF,\PP)$ which hosts  a \cd{} process $L$ with
distribution $\PP_L$, and  consider only right-continuous and complete augmentation
$\bar{\FFF}^L_+$ of the natural filtration $\FFF^L$, generated by $L$.
Let $\sun$ denote the set of all
progressively-measurable $k$-dimensional processes with values in $[0,n]^k$.
For $u\in\sun$, all components of the process $A=\int_0^{\cdot} u(t)\, dt$
are Lipschitz continuous with the Lipschitz constant
not exceeding $n$. Conversely, each adapted process with such Lipschitz
paths admits  a similar representation.  This correspondence allows us to
pose the $n$-th capped monotone follower problem either over the set of process
$\sun$ or over the appropriate admissible set
$\sAn=\sets{\int_0^{\cdot} u(t)\, dt}{u\in\sun}$.
Their (strong) value functions are then defined by
\begin{equation} \label{equ:Vn}
  \Vn = \inf_{A\in \sAn} \ee{C(L,A)} = \inf_{u\in\sun} J(u)\ewhere
  J(u)=\ee{C(L,\textstyle\int_0^{\cdot} u)}.
\end{equation}
Each $A\in\sAn$ is $\bFl_+$-adapted and, therefore, strongly admissible, in the
sense of Definition
\ref{def:adm-capped-contr}.
In particular, $\Vn \geq V$, for all $n$. Also, noting that the  polynomial-growth assumption implies that $\EE[ C(L,A)]<\infty$, for  each bounded $A$,
we have $ \Vn<\infty$, for all $n\in\N$, and, consequently, $V<\infty$.

For readability, we split the remainder of the proof into several subsections.
\subsubsection{Existence in the prelimit}
Let $\ltp$ denote the space of all ($\ld\otimes\PP$-equivalence classes)
of $\bar{\FFF}^L_+$-progressively-measurable processes $u$
on $[0,T]\times \PP$ with
\begin{align*}
  \norm{u}_{\ltp}=\ee{\textstyle\izT \abs{u(t)}^2\, dt}^{1/2}<\infty.
\end{align*}
\begin{proposition}
  \label{pro:capped-minimizer}
  The infimum in \eqref{equ:Vn} is attained at some
  $\un\in\sun$.
\end{proposition}
\begin{proof}
  We proceed in the standard way, using the so-called ``direct method''.
  Let $\seqk{u}\subset \sun$ be a minimizing sequence, i.e., $J(u_k)\downto \Vn$.
  Since $\sun$ is bounded in $\ltp$, the Banach-Sachs theorem implies that we can extract a subsequence whose Ces\' aro sums
  (still denoted by $\seqk{u}$) converge strongly towards some
  $\un\in \ltp$. Furthermore, given that $\sun$ is closed and convex, we have
  $\un\in\sun$, as
  well. Thanks to the convexity of $J$, which is inherited
  from $C$, $\seqk{u}$ remains a minimizing sequence. Hence, to show that
  $\un$ is
  the minimizer, it will be enough to
  establish lower semicontinuity of $J$ on $\sun$ which is, in turn, a direct
  consequence of Fatou's lemma.
\end{proof}

\subsubsection{A version of the Pontryagin FBSDE}
Having established the existence in the (strong) capped monotone follower
problem, for each $n\in\N$ we pick and fix a minimizer $\un$  as in Proposition
\ref{pro:capped-minimizer} and turn to a capped version of the
Pontryagin FBSDE. We state it in a very weak form
(namely, as Proposition \ref{pro:206B})
which will, nevertheless
suffice to establish the validity of the full Pontryagin FBSDE in the
limit. The following notation will be used throughout:
\begin{align*}
  \begin{aligned}
    \Awn_t & = \int_0^t \un_s\, ds,            &
    \Nwn_t & =\izt \nabla h ( L_s, \Awn_s)\,ds &
    \Fwn_t &= \int_0^t f(s)\, d\Awn_s,
  \end{aligned}
\end{align*}
as well as
\begin{align*}
  \begin{aligned}
    \Mwn_t & =\Beec{\nabla g (L_T,\Awn_T)+\Nwn_T}{\bFl_{t+}}, &
    \Ywn_t &= f(t) + \Mwn_t - \Nwn_t,
  \end{aligned}
\end{align*}
all taken in their \cd{} versions. We note immediately that, thanks to the
polynomial-growth condition, all the integrals above are well
defined, and that $\Ywn$ is the optional projection of $\da C_A(L,\Awn)$ onto
$\bar{\FFF}^L_+$.
\begin{proposition}
  \label{pro:206B} For $n\in\N$, we have
  \begin{equation}
    \label{equ:21DA}
    \begin{split}
      n \ee{ \izT (\Ywn_t)^-\, dt} =
      -\ee{\izT \Ywn_t\, d\Awn_t}
    \end{split}
  \end{equation}
  and
  \begin{equation}
    \label{equ:21DB}
    \begin{split}
      \lim_{n\to\infty} \ee{ \int_0^T (\Ywn_t)^-\, dt} =0.
    \end{split}
  \end{equation}

\end{proposition}
\begin{proof}
  Given $v\in\sun$ and $\eps \in [0,1]$ we set
  $B=\int_0^{\cdot} v_t\, dt$ and
  define
  \begin{align*}
    \Ae =\Awn+\eps(B-\Awn)\in\sAn
  \end{align*}
  Since $C(L,\Awn)\in\lone$, the optimality of $\un$ implies
  that
  \begin{align*}
    0 \geq \EE[ C(L,\Awn)] - \EE[ C(L,\Ae)] \geq \eps\EE[ \scl{
    \da C(L,\Ae )}{\Awn-B}],
  \end{align*}
  We let $\eps \downto 0$ and use the dominated convergence theorem to conclude that
  \begin{align}
    \label{equ:var}
    \begin{split}
    \ee{\izT  (\Ywn_t)^+ (\un_t-v_t)\ dt} \leq
    \ee{\izT  (\Ywn_t)^- (\un_t-v_t)\ dt}.
    \end{split}
  \end{align}
  Setting $v= n \inds{\Ywn\leq 0}$ yields
  \begin{align}
    \label{equ:var-2}
    \begin{split}
    \ee{\izT  (\Ywn_t)^+ \un_t\ dt} &\leq
    \ee{\izT  (\Ywn_t)^- (\un_t-n )\ dt}. \\
    \end{split}
  \end{align}
  Since the left-hand side of \eqref{equ:var-2} is nonnegative and the
  right-hand side nonpositive, we conclude that both of them vanish, which,
  in turn, directly implies \eqref{equ:21DA}.

  To show \eqref{equ:21DB} we use the inherited subgradient property of $\Ywn$ and \eqref{equ:21DA} to obtain
  \begin{equation*}
    \begin{split}
      0 &\leq  \ee{ C(L,\Awn)} \leq \ee{ C(L,0) } +  \ee{\int_0^T \Ywn_u\, d\Awn_u} \\ &=
      \ee{C(L,0)} - n \ee{ \izT (\Ywn_t)^-\, dt}.\qedhere
    \end{split}
  \end{equation*}
\end{proof}


\subsubsection{Relative compactness in the Meyer-Zheng topology}

Our next step is to pass to the limit, as $n\to\infty$, in the Meyer-Zheng
convergence and show that the limiting law satisfies the weak FBSDE
\eqref{def:FBSDE}.  The reader will find a short recapitulation of the pertinent known
results on the Meyer-Zheng convergence (minimally modified to fit our needs)
in subsections \ref{sse:pseudopath}, \ref{sse:meyer-zheng} and
\ref{sse:compactness} of Appendix \ref{app:Appendix}.

In the sequel,
$\sq{\tPn}$ denotes the sequence of laws of the  triplets $(L, \Awn, \Mwn)$ on $\sD^{d+2k}$.
\begin{proposition}
  \label{pro:rel-com}    For each $p\geq 1$, we have
  \begin{align}
    \label{equ:lpbound}
    \sup_n \bnorm{\Awn_T}_{\lpee}<\infty,
  \end{align}
  and    the sequence $\sq{\tPn}$ is relatively
  compact in the Meyer-Zheng topology on $\prob^{d+2k}$.
\end{proposition}
\begin{proof} Since the distribution of first component $L$ does not depend on $n$, by Theorem \ref{thm:MZ} , it will be enough to establish that
  \begin{align*}
    \sup_{n\in\N} \Var^{\PP_n}[A]<\infty \eand \sup_{n\in\N} \Var^{\PP_n}[M]<\infty,
  \end{align*}
  where $\Var^{\PP_n}$ denotes the
  conditional variation (in the quasimartingale sense, as
  defined in \eqref{equ:qvar}, below).
  Moreover, given that all $\Awn$ are nondecreasing, and all $\Mwn$ are
  martingales, relative compactness will follow once we show that
  \begin{align*}
    \sup_n \EE[ \Awn_T]<\infty \eand
    \sup_n \EE[ \sabs{\Mwn_T}] < \infty,
  \end{align*}
  for which  - thanks to our polynomial-growth assumption - it will
  suffice to establish \eqref {equ:lpbound}.
  In order to do that, for $n\in\mathbb N$ and $r\geq 0$  define
  $u^{[n];r}_t=\un_t1_{\{\Awn_t<r\}}$, so that
  \begin{align*}
    A^{[n];r}_t=\izt u^{[n];r}_sds=\Awn_{t\wedge T^{[n]}(r)},
  \end{align*}
  where $\Twn(r)=\inf\sets{t\in[0,T]}{\Awn_t\geq
  r}\in[0,T]\cup\{\infty\}$.
  By the sub-optimality of $u^{[n]; r}$ we have
  \begin{multline*}
    \ee{\izT f(t) u^{[n]; r}_t\,dt+\izT h(L_t,A^{[n]; r}_t)\,dt + g(L_t,A^{[n]; r}_T)}  \\ \geq\ee{\izT f(t) \un_t\,dt+\izT h(L_t,\Awn_t)\,dt +g(L_t,\Awn_T)},
  \end{multline*}
  so that
  \begin{multline*}
    \ee{\int_{T\wedge\Twn( r)}^T f(t)\un_t\,dt} \\ \leq\ee{\int_{T\wedge\Twn( r)}^T h( L_t,  r)-h(L_t, \Awn_t)\,dt + \left(g(L_t, r)-g(L_t,\Awn_T)\right)1_{\{\Awn_T> r\}}}.
  \end{multline*}
  Since $f$ is positive and componentwise bounded away from zero (say, by
  $c>0$), and $h$, $g$ are nonnegative and convex in their second argument, we have
  \begin{align*}
    \ee{\int_{T\wedge\Twn( r)}^T f(t)\un_tdt} \geq c\ee{(\Awn_T-
    r)1_{\{\Awn_T> r\}}},
  \end{align*}
  as well as, on $\set{\Awn_T> r}$,
  \begin{equation*}
    \label{equ:4C97}
    \begin{split}
      \int_{T\wedge\Twn( r)}^T h( L_t ,  r)
      &\leq \int_{T\wedge\Twn( r)}^T h(
      L_t,\Awn_t)\,dt+\izT h( L_t,0)\,dt\ \text{ and } \\
      g(L_t, r) &\leq g(L_t,\Awn_T)+g(L_t,0)
    \end{split}
  \end{equation*}
  It remains to apply Lemma \ref{lem:XY} with $X=\babs{\Awn_T}$ and
  $Y=\izT h(L_t,0)\,dt + g(L_t,0)$, to conclude that $\sq{\Awn_T}$ is
  bounded in $\lpee$, for each $p\geq 0$.
\end{proof}

\subsubsection{The Meyer-Zheng limit and its first properties}
Having established the relative compactness of the sequence $\sq{\tPn}$, we
select one of its limit points $\tPP^*$. By passing to a subsequence, if
necessary, we may assume that $\tPn\to\tPP^*$ in the Meyer-Zheng topology.

\begin{proposition}\label{pro:wadm}
  $\tPP^*_{(L,A)}$ is (weakly) admissible.
\end{proposition}
\begin{proof}
  Since the first components $L$  have the same law under each $\tPn$ (namely $\PP_0$),  it is clear that the same remains true in the limit.
  To establish the requirement (2) of Definition \ref{def:adm-contr}, we pick
  $m\in\N$, two continuous and bounded functions $F: (\R^k)^m \to \R$ and  $H:(\R^d)^m \to \R$, as well as
  a $C^{\infty}_c(\R^d)$-function $G$. Thanks to the admissibility of each
  $\tPn$, for
  each $n\in\N$ and all  $t<s_1<\dots < s_m\leq T$, we have
  \begin{align*}
    \eecq{\tPn}{ F\,  G(L_T)}{\sF^L_{t+}}=
    \eecq{\tPn}{ F       }{\sF^L_{t+}}
    \eecq{\tPn}{ G(L_T) }{ \sF^L_{t+}},
  \end{align*}
  where $F=F(A_{s_1},\dots, A_{s_m})$.
  Since $\tPn_L=\PP_L$ and thanks to first assumption  of Theorem \ref{thm:approx}, for all $n\in\N$ we have
  \begin{align*}
    \eecq{\tPn}{ G(L_T) }{ \sF^L_{t+}} = G^*(L_t),\  \tPn-\as,
  \end{align*}
  for some $G^* \in C_b(\R^k)$. Thus, for $0\leq r_1<\dots<r_m\leq t$, we have
  \begin{align*}
    \eeq{\tPn}{ F\,  G(L_T) H}=
    \eeq{\tPn}{ F\, G^*(L_t) H},\ \Pn-\as,
  \end{align*}
    where $H=H(L_{r_1},\dots, L_{r_m})$. Thanks to 
      Theorem \ref{thm:MZ-fin-dim}, after another passage to a subsequence, 
              there exists a full-measure subset $\sT$ of
                $[0,T]$, which includes $T$, such that
                $\PP^n$-finite-dimensional distributions
                with indices in $\sT$ converge towards the $\PP$-finite-dimensional
                distributions. Hence, 
              if $r_1<\dots<r_m$, $t$ and $s_1<\dots < s_m$ belong to such $\sT$, we
                have
  \begin{align*}
    \eeq{\tPP^*}{ F\,  G(L_T) H}=
    \eeq{\tPP^*}{ F\, G^*(L_t) H}.
  \end{align*}
  It follows that, for $t\in \sT$, we have
  \begin{equation}
    \label{equ:88}
    \eecq{\tPP^*}{ F G(L_T) }{ \sF^L_t} = \eecq{\tPP^*}{F}{\sF^L_t} \eecq{\tPP^*}{G(L_T)}{\sF^L_{t}},
    \ \tPP^*-\as,
  \end{equation}
  for all $F,G$. It is a part of our assumptions that a version of the
    Blumenthal's $0-1$-law holds. By Proposition \ref{pro:wadm},
      $\tPP^*_L= \PP^0_L$; it  
        follows that $\sigma$-algebras $\sF^L_{t}$ and $\sF^L_{t+}$ coincide
  $\tPP^*$-a.s., for each $t$.
  Moreover,
  both sides of the equality in \eqref{equ:88} above admit right-continuous versions, so
  it remains to use the density of $\sT$ in $[0,T]$ to conclude that $\tPP^*$ is also weakly admissible.
\end{proof}

Next, we couple the probability measures $\sq{\tPn}$ and $\tPP^*$ on the same probability space.
\begin{lemma}
  \label{lem:limk}
  There exists a probability space and on it a sequence $\sq{(\An,\Ln,\Mn)}$ of
  $\sD^{d+2k}$-valued random elements, as well as an $\sD^{d+2k}$-valued random element $(A,L,M)$ such that
  \begin{enumerate}
    \item the law of $(\Ln,\An,\Mn)$ is $\tPn$, and the law of $(L,A,M)$ is $\tPP^*$, and
    \item For almost all $\omega\in\Omega$, we have
      \begin{align*}
        (\Ln_T(\omega),\An_T(\omega),\Mn_T(\omega)) \to
        (L_T(\omega),A_T(\omega), M_T(\omega))
      \end{align*}
      as well as
      \begin{align*}
        (\Ln_t(\omega),\An_t(\omega),\Mn_t(\omega)) \to
        (L_t(\omega),A_t(\omega), M_t(\omega))
      \end{align*}
      in (Lebesgue) measure in $t$.
  \end{enumerate}
\end{lemma}
\begin{proof}
  The first step is use Dudley's extension (see \cite{Dud68}, Theorem 3., p.~1569) of the Skorokhod's representation theorem to transform the Meyer-Zheng convergence to an
  almost-sure convergence in the pseudopath topology. Indeed, the original
  theorem of Skorokhod cannot be applied directly since the canonical space
  $\sD^{d+2k}$, together with the pseudopath topology is not Polish.  Next, a
  minimal adjustment of a result of Dellacherie (see Lemma 1., p.~356 in
  \cite{MeyZhe84}) states that the pseudopath topology and the topology of the
  convergence in the sum $\ld+\delta_{T}$ of the Lebesgue measure $\ld$ on
  $[0,T]$ and the Dirac mass $\delta_T$ on $\set{T}$ coincide.
\end{proof}

On the probability space of Lemma \ref{lem:limk}, we define the  sequences
\begin{align*}
  \Nn &= \int_0^{\cdot} \nabla h (\Ln_t,\An_t)\, dt, & \Fn_{\cdot} &= \int_0^{\cdot} f(t)\, d\An_t, \\
  \intertext{as well as}
  N &= \int_0^{\cdot } \nabla h (L_t,A_t)\, dt,  & F &= \int_0^{\cdot} f(t)\, dA_t,  &
\end{align*}
Using the polynomial-growth assumptions and the $\lpee$-boundedness of $\sq{\An}$
we see immediately that
\begin{align*}
  \Nn\to N \text{ in } \lone(\ld\otimes \PP),  \eand \Mn_T\tolone
  M_T.
\end{align*}
To deal with $\sq{\Fn}$, we can use an argument completely analogous to that
in the last part of the proof of Theorem \ref{thm:exist} (with $K$ replaced by
$[0,T]\setminus \sT$). Indeed, together with the $\lpee$-boundedness of
$\sq{\An_T}$, for all $p\geq 1$, it yields that
\begin{equation}
  \label{equ:FntoF}
  \Fn_T \tolone F_T.
\end{equation}
\subsubsection{A passage to a limit in the Pontryagin FBSDE}
We define $\Yn = f + \Mn - \Nn $ so that
\begin{align*}
  \Yn \to Y=f+ M - N \text{ in } \lone(\ld\otimes \PP).
\end{align*}
Thus,
\begin{equation*}
  \label{equ:352C}
  \begin{split}
    \EE[ \izT Y_t^-\, dt] & = \lim_n \EE[ \izT (\Yn_t)^-\,
    dt] = \lim_n \EE[ \izT (\Ywn_t)^-\, dt] = 0,
  \end{split}
\end{equation*}
where the last equality follow directly from equation \eqref{equ:21DB} of
Proposition \ref{pro:206B}. Consequently, by right continuity,
\begin{align}\label{equ:Ypos}
Y_t\geq 0\eforall t\in [0,T].
\end{align}
Next, we observe that, by Lemma \ref{lem:limk} and equation \eqref{equ:FntoF}, we have
\begin{align*}
  \EE[ C(L,A)] =  \lim_n \EE[ C(\Ln,\An)]=\inf_n \EE[ C(\Ln, \An)].
\end{align*}
Therefore, for each $n\in\N$, we have
\begin{align*}
  0 \leq \EE[ C(\Ln, \An)] - \EE[ C(L,A)]  =: K_n + I_n,
\end{align*}
where
\begin{align*}
K_n= \EE[ C(\Ln,\An)] - \EE[ C(\Ln,A)] \eand
I_n=\EE[ C(\Ln,A) - C(L, A)].
\end{align*}
By convexity of $h$ and $g$ and integration by parts  we have
\begin{align*}
  K_n &\leq \EE[ \scl{ \da C (\Ln,\An)}{ \An-A}]
  = \ee{\Fn_T-F_T}+\\ &\qquad+\ee{\izT \nabla h (\Ln_t, \An_t) (
  \An_t-A_t)\,dt+\nabla
  g (\Ln_T,\An_T) (\An_T - A_T)}\\
  & =
  \EE[\int_0^T \Yn_t\, d\An_t] - R_n
\end{align*}
where
\begin{align*}
  R_n=\EE[ F_T + \int_0^T \nabla h (\Ln_t, \An_t) A_t\, dt + \nabla g (\Ln_T,\An_T) A_T].
\end{align*}
By equation \eqref{equ:21DA} of Proposition \ref{pro:206B}, we then have
\begin{align*}
  K_n \leq -n \EE[ \int_0^T (\Yn_t)^-\, dt] - R_n \leq -R_n.
\end{align*}
On the other hand, thanks to the growth assumptions,
the family $\sq{ C(\Ln,A) - C(L,A)}$ is
uniformly integrable. By the continuity of $g$ and $h$ in the $l$-argument, we
have $C (\Ln,A)\to C (L,A)$ a.s., so $I_n\to 0$, as $n\to\infty$.
It follows that $\liminf R_n \leq  0$, and, therefore,
\begin{equation}
  \label{equ:flez}
  \EE[ F_T + \int_0^T \nabla h (L_t, A_t) A_t\, dt + \nabla g (L_A,A_T) A_T] \leq 0.
\end{equation}
Next we investigate the martingale properties of the third component process
$M$, in the spirit of the martingale-preservation
property of the Meyer-Zheng convergence (see Theorem 11., p.~368 in \cite
{MeyZhe84} ).
On the filtered probability space of the capped problem (i.e.,
of subsection \ref{sse:proof-approx}), the process
$\Mwn$ is a martingale, and $\Awn$ is adapted with respect to the
augmented filtration generated by $L$. Thus, we have
\begin{align*}
  \ee{ \Mn_t \vp\Big((\Ln_{t_i},\An_{t_i}, \Mn_{t_i})_{1\leq i \leq k}\Big)}=
  \ee{\Mn_T \vp\Big((\Ln_{t_i},\An_{t_i}, \Mn_{t_i})_{1\leq i \leq k}\Big)}
\end{align*}
for each $k\in\N$, a continuous bounded function $\vp:\R^{d+2k}\to\R$
and any choice of $0\leq t_1<t_2<\dots < t_k\leq t$.  It follows that, with
$\sT$ as in Theorem \ref{thm:MZ-fin-dim}, that
\begin{align}\label{equ:MMart}
  \EE[ M_T|\sF^{L,A,M}_t] = M_t,\text{ a.s., for $t\in \sT$,}
\end{align}
and, then, by the right-continuity of the paths of $M$,
that $M$ is an $\FFF^{(L,A,M)}$-
martingale. The inequality \eqref{equ:flez} implies that after another
round of integration by parts - we have
\begin{align} \label{equ:YdA}
  \int_0^T Y_t\, dA_t \leq 0,\text{ a.s.}
\end{align}
It remains to aggregate the above results to conclude that the (law)
of the triplet $(L,A,Y)$ is a weak solution of the Pontryagin FBSDE
(Definition \ref{def:FBSDE}). Part (1) is exactly the content of Proposition
\ref{pro:wadm}, while part (2) follows from \eqref{equ:Ypos} and \eqref
{equ:YdA}.
Finally
(3) is simply a restatement of the martingale property of the process $M$,
established after \eqref{equ:MMart} above.
Theorem \ref{thm:max-equiv}, part (2) now allows us to conclude that the law of
the pair
$(L,A)$ is a solution to the monotone follower problem.

\appendix
\section{Auxiliary results}
In this appendix we gather several results that are used in the body
of the paper. They either admit hard-to-locate standard proofs, or are minimal extensions of the known results; we state them here, and supply proofs, for completeness sake.
\label{app:Appendix}
\subsection{Coupling of weakly admissible controls}
\label{sse:coupling}
We start simple coupling lemma based on a standard use of regular
conditional probabilities. It is used in proofs of Theorem
\ref{thm:exist} and Theorem \ref{thm:approx} above.
\begin{lemma}[Coupling]
  \label{lem:coupling}
  For $d,k,l\in\N$, let  $\PP\in\prob^{d+k}(L,Q)$ and
  $\PP'\in\prob^{d+l}(L',R')$
  be such that $\PP_{L} = \PP'_{L'}$. Then,  there exists
  a probability measure
  $\bar{\PP}\in\prob^{d+k+l}(\bar{L},\bar{Q},\bar{R})$,
  denoted by $\PP\ofl\PP'$ such that
  \begin{enumerate}
    \item $\bar{\PP}_{\bar{L},\bar{Q}} = \PP_{L,Q}$,
    \item $\bar{\PP}_{\bar{L},\bar{R}} = \PP'_{L',R'}$, and
    \item $\bar{Q}$ and $\bar{R}$ are
      $\bar{\PP}$-conditionally independent, given $\bar{L}$.
  \end{enumerate}
\end{lemma}
\begin{proof}
  The space $\sD^d(L,Q)$ is a Borel space, so there exists a regular conditional
  distribution (r.c.d.)
  \begin{align*}
    \mu: \sD^{d}(L,Q)\times \sB(\sD^k) \to [0,1],\ \mu(x,B)
    = \PP[ Q\in B| L=x],
  \end{align*}
  for $Q$, given $L$ under $\PP$.
  Similarly, let
  $\mu': \sD^d(L',R') \times \sB(\sD^l)\to [0,1]$ denote the $\PP'$-r.c.d.~of
  $R$ given $L'$ and
  let $\rho$ denote the the product kernel $\rho:\sD^d \times
  \sB(\sD^{k+l}) \to [0,1]$, given by
  \begin{align*}
    \rho(x, B) = (\mu(x,\cdot)\otimes \nu(x,\cdot))(B),\efor x\in\sD^d \eand
    B\in\sB(\sD^{k+l}).
  \end{align*}
  We define $\bar{\PP}$ as
  the (Ionescu-Tulcea-type) product $\PP_L \otimes \rho$ of the measure $\PP_L$
  and the kernel $\rho$, i.e., the probability measure given by
  \begin{align*}
    \bar{\PP}[C] = \int_{x\in \sD^d} \int_{(q,r)\in \sD^{k+l}}
    \ind{C}(x,q,r)\, \rho(x,dq,dr)\, \PP_L(dx),
  \end{align*}
  for $C\in\sB(\sD^{d+k+l})$.
  The reader will readily check that so defined, $\bar{\PP}=\PP\ofl\PP'$ satisfies
  all three conditions in the statement.
\end{proof}
An immediate application of Lemma \ref{lem:coupling} is the
following
\begin{lemma}
  \label{lem:prob-space}
  Let $\seq{\PP}$ be a sequence in $\sA$.  Then, there exists
  a probability space and, on it,
  \cd{} processes $\prf{L_t}$, $\prf{\An_t}$, $n\in\N$,
  such that the joint law of $(L,\An)$ is $\PP_n$, for each
  $n\in\N$,
  and $\sq{\An}$ are independent, conditionally on $L$.
\end{lemma}
\begin{proof}
  We can think of the required sequence $L,A^{(1)}, A^{(2)} ,\dots$ as a
  stochastic process with values in $\sD^k$ (and $\sD^d$ for its first
  component). Using the information on the joint distributions and the
  requirement of conditional independence from the statement, we can apply
  Lemma \ref{lem:coupling} repeatedly to construct its (consistent) family of
  finite-dimensional distributions. The target spaces $\sD^d$ and $\sD^k$ are
  Polish, so the sought-for probability space $(\Omega,\FF,\PP)$
  can now be constructed by using Kolmogorov's extension theorem.
\end{proof}
\subsection{An $\lpee$ estimate}
\begin{lemma}\label{lem:XY}
  Given $p\geq1$, suppose that  $X\in\lone_+$ and $Y\in\lpee_+$ satisfy
  \begin{align}\label{XY}
    \EE[(X-r)^+]\leq\EE[Y1_{\{X>r \}}], \quad for\ all\ r \geq0.
  \end{align}
  Then, $X\in\lpee$ and $\lnorm{X}_p\leq p\lnorm{Y}_p$.
\end{lemma}
\begin{proof}
  The conclusion clearly holds for $p=1$: it suffices to substitute $r=0$
  into \eqref{XY}. For $p>1$, multiplying both sides of \eqref{XY} by
  $(p-1)r^{p-2}$ and integrating in $r$ over $[0,M]$, for $M>0$,  yields
  \begin{align*}
    \EE[Y(X\wedge M)^{p-1}]&=
    \EE\left[Y\int_0^M
    (p-1)r^{p-2}1_{\{X>r\}}dr\right]\\
    &\geq\EE\left[\int_0^M (p-1)r^{p-2}(X-r)^+ dr\right]\geq
    \frac{1}{p}\EE[(X\wedge M)^p].
  \end{align*}
  It remains to apply H\"older's inequality to obtain
    \[
    \frac{1}{p}\EE[(X \wedge M)^p]\leq\EE[Y(X \wedge
    M)^{p-1}]\leq\lnorm{Y}_p\lnorm{X\wedge M}_p^{p-1},
    \]
    which, after dividing both sides by  $\norm{X\wedge M}_{p}^{p-1}$, and letting
    $M\to\infty$, completes the proof. 
\end{proof}
\subsection{The pseudopath topology}
\label{sse:pseudopath}
The topology $\tau_{pp}$ we consider on $\sD^N$ is a following minimal
modification of the pseudopath topology introduced in \cite{MeyZhe84}.

A path $x\in \sD^N$ can be identified with its
\define{pseudopath}, i.e., a finite
measure on the product $[0,T]\times \R^N$, obtained as a push-forward of
the ``reinforced'' Lebesgue measure $\Leb+\delta_{\set{T}}$
on $[0,T]$,  where $\delta_{\set{T}}$ denotes the Dirac mass at $\set{T}$,
via the map
\begin{align*}
  [0,T]\ni t\mapsto (t,x(t)) \in [0,T]\times \R^N.
\end{align*}
With such an identification, the trace of the topology of weak convergence
of measures is induced on $\sD^N$;  we call it the \define{pseudopath
topology} and denote by $\tau_{pp}$. It is shown in \cite[Lemma 1, p.~365]{MeyZhe84} - we modify
this result (and all others) minimally to fit our setting - that the pseudopath topology is
metrizable and that, for a sequence $\seq{x}$
in $\sD$,  we have $x_n \topp x\in\sD$, where $\topp$ denotes the
convergence in the pseudopath topology, if and only if
\begin{equation}\label{equ:ppath-char}
  \begin{split}
    &x_n(T) \to x(T)\eand
    \izT b(s,x_n(s))\, ds \to \izT b(s,x(s))\, ds,
  \end{split}
\end{equation}
for all continuous and bounded functions $b:[0,T]\times \R^N\to \R$. Finally,
we mention a result due to Dellacherie (see \cite{MeyZhe84}, Lemma 1, p.~356)
which simply states that the convergence in the pseudopath topology and the
convergence in the measure $\ld +\delta_{\set{T}}$ coincide.

\subsection{The Meyer-Zheng convergence}
\label{sse:meyer-zheng}
Using the pseudopath topology $\tau_{pp}$ on $\sD^N$, one can define the
\define{Meyer-Zheng topology} on $\prob^N$  as the topology of weak
convergence of probability measures on the topological space
$(\sD^N,\tau_{pp})$.
Like the pseudopath topology $\tau_{pp}$ on $\sD^N$,
the Meyer-Zheng topology on $\prob$ is metrizable, but not necessarily
Polish (see p.~372 in \cite{MeyZhe84}); the convergence in the
Meyer-Zheng topology is denoted by $\tomz$.
As shown in \cite{MeyZhe84},  the Borel $\sigma$-algebra generated by the
pseudopath topology $\tau_{pp}$ coincides with the canonical
$\sigma$-algebra on $\sD^N$,
i.e., the one induced by the coordinate maps or, equivalently,
by the Skorokhod topology. Moreover, the set of all pseudopaths, denoted by $\Psi$, under $\tau_{pp}$ is Polish.

We note the following (minimal extension) of a useful consequence of the Meyer-Zheng convergence
( see \cite{MeyZhe84}, Theorem 5., p.~365):
\begin{theorem}[Meyer and Zheng, 1984]
\label{thm:MZ-fin-dim}
Let $\sq{\PP^n}$ be a sequence
of probability measures on $\sD^N$ such that that $\PP^n\to \PP$ in the Meyer-Zheng sense.
Then there exists a  subset $\sT\subseteq [0,T]$ of full
Lebesgue measure, containing $T$, such that the $\PP^n$-finite-dimensional distributions
with indices in $\sT$ of the coordinate
process converge to the corresponding finite-dimensional distributions under $\PP$, perhaps after a passage to a subsequence.
\end{theorem}

\subsection{A criterion for compactness}
\label{sse:compactness}
One of the reasons
the Meyer-Zheng topology proved to be quite useful in
probability theory and optimal stochastic control is a simple
characterization of compactness it affords. Unlike the Skorokhod topology,
where compactness needs a stronger form of equicontinuity, the subsets of $\prob^N$ are
Meyer-Zheng-compact as soon as they are suitably bounded.
The following result is a compilation of two statements in
\cite{MeyZhe84}, namely Theorem 4., p.~360, and Theorem 5., p.~365,
minimally adapted to fit our setting.
We remind the reader that an adapted stochastic process $X$, defined on a
filtered measurable space $(\Omega,\sF,\prf{\sF_t})$ is
said to be a \define{quasimartingale} under the probability measure $\PP$
if $X_t\in\lone(\PP)$, for all $t\in [0,T]$ and $\Var^{\PP}[X]<\infty$, where
\begin{align}
  \label{equ:qvar}
  \Var^{\PP}[X]= \sup \sum_{j=1}^m
  \Beeq{\PP}{ \Babs{\eeq{\PP}{
  X_{t_j}-X_{t_{j-1}}\big|\sF_{t_j}}}}
  +\eeq{\PP}{\abs{X_T}},
\end{align}
and the supremum is taken over all  partitions $0=t_0< \dots <
t_m=T$, $m\in\N$,  of $[0,T]$.
\begin{theorem}[Meyer and Zheng, 1984]
\label{thm:MZ}
Let $\seq{\PP}$ be a sequence of probability measures on $\sD^N$ (equipped with the filtration generated by the coordinate maps) with the
property that each coordinate process $\prf{X^i_t}$,
$i=\ft{1}{N}$, is a $\PP_n$-quasimartingale
for each $n\in\N$ and
\begin{align*}
  \sup_{n\in\N} \Var^{\PP_n}[X^i]<\infty,
  \text{ for all } i=\ft{1}{N}.
\end{align*}
Then, there exists a subsequence $\sqk{\PP_{n_k}}$
of $\seq{\PP}$ and $\PP\in\prob$ such that
$\PP_{n_k}\tomz \PP$ in the Meyer-Zheng topology.
\end{theorem}

\begin{remark}
The condition $\sup_n \Var^{\PP_n}[X^i]<\infty$ is easy to check if $X^i$ is a
$\PP_n$ martingale, for each $n\in\N$. Indeed, in that case
$\Var^{\PP^n}[X^i]=\eeq{\PP_n}{\abs{X^i_T}}$, with its boundedness being
equivalent to uniform $\lone$-boundedness of the process $X^i$ under all
$\seq{\PP}$.

Similarly, if $X^i$ happens to be a process of finite variation,
$\Var^{\PP_n}[X^i]$ is bounded from above by a (constant multiple) of the
expected total variation of $X^i$. In particular, if $X^i$ is
nonnegative and nondecreasing
under all $\PP^n$, the condition we are looking for is exactly the same
as in the martingale case: $\sup_n \eeq{\PP_n}{\abs{X^i_T}}<\infty$.
\end{remark}

\def\ocirc#1{\ifmmode\setbox0=\hbox{$#1$}\dimen0=\ht0 \advance\dimen0
  by1pt\rlap{\hbox to\wd0{\hss\raise\dimen0
  \hbox{\hskip.2em$\scriptscriptstyle\circ$}\hss}}#1\else {\accent"17 #1}\fi}
  \ifx \cprime \undefined \def \cprime {$\mathsurround=0pt '$}\fi\ifx \k
  \undefined \let \k = \c \fi\ifx \scr \undefined \let \scr = \cal \fi\ifx
  \soft tundefined \def \soft
  {\relax}\fi\def\ocirc#1{\ifmmode\setbox0=\hbox{$#1$}\dimen0=\ht0
  \advance\dimen0 by1pt\rlap{\hbox to\wd0{\hss\raise\dimen0
  \hbox{\hskip.2em$\scriptscriptstyle\circ$}\hss}}#1\else {\accent"17 #1}\fi}
  \ifx \cprime \undefined \def \cprime {$\mathsurround=0pt '$}\fi\ifx \k
  \undefined \let \k = \c \fi\ifx \scr \undefined \let \scr = \cal \fi\ifx
  \soft tundefined \def \soft {\relax}\fi


\begin{thebibliography}{MPY94}
\expandafter\ifx\csname urlstyle\endcsname\relax
  \providecommand{\doi}[1]{doi:\discretionary{}{}{}#1}\else
  \providecommand{\doi}{doi:\discretionary{}{}{}\begingroup
  \urlstyle{rm}\Url}\fi

\bibitem[AM03]{AntMa03}
\textsc{Antonelli, F. and Ma, J.}
\newblock Weak solutions of forward-backward {SDE}'s.
\newblock \emph{Stochastic Anal. Appl.}, 21(3):493--514, 2003.

\bibitem[Ban05]{Ban05}
\textsc{Bank, P.}
\newblock Optimal control under a dynamic fuel constraint.
\newblock \emph{SIAM J. Control Optim.}, 44(4):1529--1541 (electronic), 2005.

\bibitem[BC67]{BatChe67}
\textsc{Bather, J. and Chernoff, H.}
\newblock Sequential decisions in the control of a spaceship.
\newblock In \emph{Proc. {F}ifth {B}erkeley {S}ympos. {M}athematical
  {S}tatistics and {P}robability ({B}erkeley, {C}alif., 1965/66), {V}ol. {III}:
  {P}hysical {S}ciences}, pages 181--207. Univ. California Press, Berkeley,
  Calif., 1967.

\bibitem[BC09]{BaiCri09}
\textsc{Bain, A. and Crisan, D.}
\newblock \emph{Fundamentals of stochastic filtering}, volume~60 of
  \emph{Stochastic Modelling and Applied Probability}.
\newblock Springer, New York, 2009.

\bibitem[BR01]{BanRie01a}
\textsc{Bank, P. and Riedel, F.}
\newblock Optimal consumption choice with intertemporal substitution.
\newblock \emph{Ann. Appl. Probab.}, 11(3):750--788, 2001.

\bibitem[BR06]{BudRos06}
\textsc{Budhiraja, A. and Ross, K.}
\newblock Existence of optimal controls for singular control problems with
  state constraints.
\newblock \emph{Ann. Appl. Probab.}, 16(4):2235--2255, 2006.

\bibitem[BY78]{BreYor78}
\textsc{Br{\'e}maud, P. and Yor, M.}
\newblock Changes of filtrations and of probability measures.
\newblock \emph{Z. Wahrsch. Verw. Gebiete}, 45(4):269--295, 1978.

\bibitem[CH94]{CadHau94}
\textsc{Cadenillas, A. and Haussmann, U.~G.}
\newblock The stochastic maximum principle for a singular control problem.
\newblock \emph{Stochastics Stochastics Rep.}, 49(3-4):211--237, 1994.
\newblock ISSN 1045-1129.

\bibitem[CM96]{CviMa96}
\textsc{Cvitani{\'c}, J. and Ma, J.}
\newblock Hedging options for a large investor and forward-backward {SDE}'s.
\newblock \emph{Ann. Appl. Probab.}, 6(2):370--398, 1996.

\bibitem[Dud68]{Dud68}
\textsc{Dudley, R.~M.}
\newblock Distances of probability measures and random variables.
\newblock \emph{Ann. Math. Statist}, 39:1563--1572, 1968.

\bibitem[GT08]{GuoTom08}
\textsc{Guo, X. and Tomecek, P.}
\newblock Connections between singular control and optimal switching.
\newblock \emph{SIAM J. Control Optim.}, 47(1):421--443, 2008.

\bibitem[HS95]{HauSuo95a}
\textsc{Haussmann, U.~G. and Suo, W.}
\newblock Singular optimal stochastic controls. {I}. {E}xistence.
\newblock \emph{SIAM J. Control Optim.}, 33(3):916--936, 1995.

\bibitem[Kab99]{Kab99}
\textsc{Kabanov, Y.}
\newblock Hedging and liquidation under transaction costs in currency markets.
\newblock \emph{Finance and Stochastics}, 3(2):237--248, 1999.

\bibitem[Kal02]{Kal02b}
\textsc{Kallenberg, O.}
\newblock \emph{Foundations of modern probability}.
\newblock Probability and its Applications (New York). Springer-Verlag, New
  York, second edition, 2002.

\bibitem[KS84]{KarShr84}
\textsc{Karatzas, I. and Shreve, S.~E.}
\newblock Connections between optimal stopping and singular stochastic control.
  {I}. {M}onotone follower problems.
\newblock \emph{SIAM J. Control Optim.}, 22(6):856--877, 1984.

\bibitem[MC01]{MaCvi01}
\textsc{Ma, J. and Cvitani{\'c}, J.}
\newblock Reflected forward-backward {SDE}s and obstacle problems with boundary
  conditions.
\newblock \emph{J. Appl. Math. Stochastic Anal.}, 14(2):113--138, 2001.

\bibitem[MPY94]{MaProYon94}
\textsc{Ma, J., Protter, P., and Yong, J.~M.}
\newblock Solving forward-backward stochastic differential equations
  explicitly---a four step scheme.
\newblock \emph{Probab. Theory Related Fields}, 98(3):339--359, 1994.

\bibitem[MY99]{MaJon99}
\textsc{Ma, J. and Yong, J.}
\newblock \emph{Forward-backward stochastic differential equations and their
  applications}, volume 1702 of \emph{Lecture Notes in Mathematics}.
\newblock Springer-Verlag, Berlin, 1999.

\bibitem[MZ84]{MeyZhe84}
\textsc{Meyer, P.-A. and Zheng, W.~A.}
\newblock Tightness criteria for laws of semimartingales.
\newblock \emph{Ann. Inst. H. Poincar\'e Probab. Statist.}, 20(4):353--372,
  1984.

\bibitem[MZ11]{MaZha11}
\textsc{Ma, J. and Zhang, J.}
\newblock On weak solutions of forward-backward {SDE}s.
\newblock \emph{Probab. Theory Related Fields}, 151(3-4):475--507, 2011.

\bibitem[Pra99]{Pra99}
\textsc{Pratelli, M.}
\newblock An alternative proof of a theorem of {A}ldous concerning convergence
  in distribution for martingales.
\newblock In \emph{S\'eminaire de {P}robabilit\'es, {XXXIII}}, volume 1709 of
  \emph{Lecture Notes in Math.}, pages 334--338. Springer, Berlin, 1999.

\bibitem[RS11]{RieSu11}
\textsc{Riedel, F. and Su, X.}
\newblock On irreversible investment.
\newblock \emph{Finance Stoch.}, 15(4):607--633, 2011.

\bibitem[Sch86]{Sch86}
\textsc{Schwartz, M.}
\newblock New proofs of a theorem of {K}oml{\'o}s.
\newblock \emph{Acta Math. Hung.}, 47:181--185, 1986.

\bibitem[Sch12]{Sch12}
\textsc{Schnurr, A.}
\newblock On the semimartingale nature of {F}eller processes with killing.
\newblock \emph{Stochastic Process. Appl.}, 122(7):2758--2780, 2012.

\bibitem[Ste12]{Ste12}
\textsc{Steg, J.-H.}
\newblock Irreversible investment in oligopoly.
\newblock \emph{Finance Stoch.}, 16(2):207--224, 2012.

\end{thebibliography}
\end{document}